\documentclass[11pt,a4paper]{article}
\usepackage{caption}
\usepackage{epsf,epsfig,amsfonts,amsgen,amsmath,amstext,amsbsy,amsopn,amsthm}
\usepackage{color}
\usepackage{mathrsfs}
\usepackage{tikz}
\usepackage{enumerate}
\usepackage{url}
\usepackage{enumitem}

\newtheorem{dfn}{Definition}[section]
\newtheorem{thm}{Theorem}[section]
\newtheorem{thmo}{Theorem}[section]

\newtheorem{claim}[thmo]{Claim}

\newtheorem{lem}[thm]{Lemma}

\newtheorem{prop}[thm]{Proposition}

\newtheorem{corollary}[thm]{Corollary}
\newtheorem{conjecture}[thm]{Conjecture}

\def\QED{\hfill \rule{7pt}{7pt}}

\newcommand{\dH}{{\mathcal{H}}}

\newcommand{\dE}{{\mathcal{E}}}

\newcommand{\dA}{{\mathcal{A}}}

\newcommand{\dC}{{\mathcal{C}}}
\newcommand{\dD}{{\mathcal{D}}}
\newcommand{\dF}{{\mathcal{F}}}
\newcommand{\dG}{{\mathcal{G}}}
\newcommand{\dL}{{\mathcal{L}}}

\newcommand{\dP}{{\mathcal{P}}}
\newcommand{\dR}{{\mathcal{R}}}

\newcommand{\dM}{{\mathcal{M}}}

\def\ex{\mathrm{ex}}
\newcommand{\RNum}[1]{\uppercase\expandafter{\romannumeral #1\relax}}

\newcommand{\1}{{\uppercase\expandafter{\romannumeral1}}}
\newcommand{\2}{{\uppercase\expandafter{\romannumeral2}}}
\begin{document}
\title{Some exact results on $4$-cycles: stability and supersaturation}

\author{
Jialin He\footnote{Email: hjxhjl@mail.ustc.edu.cn.}~~~~~~~
Jie Ma\footnote{Email: jiema@ustc.edu.cn.
Partially supported by NSFC grant 11622110, National Key Research and Development Project SQ2020YFA070080, and Anhui Initiative in Quantum Information Technologies grant AHY150200.}~~~~~~~
Tianchi Yang\footnote{Email: ytc@mail.ustc.edu.cn.}
}
\date{}

\maketitle

\begin{abstract}
Extremal problems on the $4$-cycle $C_4$ played a heuristic important role in the development of extremal graph theory.
A fundamental theorem of F\"uredi states that the Tur\'an number $\ex(q^2+q+1, C_4)\leq \frac12 q(q+1)^2$ holds for every $q\geq 14$,
which matches with the classic construction of Erd\H{o}s-R{\'e}nyi-S\'os and Brown from finite geometry for prime powers $q$.

Very recently, we \cite{HMY20} obtained the first stability result on F\"uredi's theorem, by showing that
for large even $q$, every $(q^2+q+1)$-vertex $C_4$-free graph with more than $\frac12 q(q+1)^2-0.2q$ edges must be a spanning subgraph of a unique polarity graph.
Using new technical ideas in graph theory and finite geometry, we strengthen this by showing that the same conclusion remains true if the number of edges is lowered to $\frac12 q(q+1)^2-\frac12 q+o(q)$.
Among other applications, this gives an immediate improvement on the upper bound of $\ex(n,C_4)$ for infinitely many integers $n$.

A longstanding conjecture of Erd\H{o}s and Simonovits states that every $n$-vertex graph with $\ex(n,C_4)+1$ edges contains at least $(1+o(1))\sqrt{n}$ 4-cycles.
In \cite{HMY20} we proved an exact result and confirmed Erd\H{o}s-Simonovits conjecture for infinitely many integers $n$.
As the second main result of this paper,
we further characterize all extremal graphs for which achieve the $\ell${'th} least number of copies of $C_4$ for any fixed positive integer $\ell$.
This can be extended to more general settings and provides enhancements on the understanding of the supersaturation problem of $C_4$.
\end{abstract}

\section{Introduction}
Given a graph $F$, we say a graph is {\it $F$-free} if it does not contain $F$ as a subgraph.
The {\it Tur\'an number} $\ex(n,F)$ of $F$ is the maximum number of edges in an $n$-vertex $F$-free graph.
Tur\'an type and related extremal problems are the central subjects of extremal graph theory.
In this paper, we focus on extremal problems on one of the basic and perhaps most influential objects in this area -- the cycle $C_4$ of length four.
(For indistinct notations appeared below, we shall refer readers to Section \ref{sec:preli}.)

Proposed by Erd\H{o}s \cite{E38} more than 80 years ago, the study of $\ex(n,C_4)$ has a rich history.
In \cite{R58} Reiman showed a general upper bound that $\ex(n,C_4)\le\frac{n}{4}(1+\sqrt{4n-3})$.
However it is known that the equality never holds by the Friendship Theorem of Erd\H{o}s, R{\'e}nyi and S{\'o}s \cite{ERS}.
One can also deduce from the proof of Reiman that if the number of edges in an $n$-vertex $C_4$-free graph is close to $\frac12 n^{3/2}$,
then almost all vertices have roughly $\sqrt{n}$ neighbors and almost all pairs of vertices have one common neighbor.
This suggests that perhaps in principle, the neighborhoods of vertices can be regarded as lines of certain projective plane.
Indeed, using orthogonal polarity graphs constructed from finite projective planes,
Erd\H{o}s-R{\'e}nyi-S\'os \cite{ERS} and Brown \cite{Br66} proved a lower bound that
\begin{equation}\label{equ:lower}
\ex(q^2+q+1,C_4)\geq \frac12 q(q+1)^2 \mbox{~ for all prime powers } q.
\end{equation}
These two results together imply an asymptotic formula that $\ex(n,C_4)=\big(\frac12+o(1)\big) n^{3/2}$.

Determining the exact value of $\ex(n,C_4)$ in general seems to be extremely difficult and far beyond reach.
On the other hand, Erd\H{o}s conjectured  (e.g. \cite{E74}) that the orthogonal polarity graph is optimal;
that is, the inequality in \eqref{equ:lower} should be replaced with an equality for all prime powers $q$.
F\"uredi \cite{Fu83} first confirmed this for $q=2^k$ in 1983, by showing
$\ex(q^2+q+1,C_4)\leq \frac12 q(q+1)^2 \mbox{ holds for all even } q.$
In 1996, F\"uredi \cite{Fur96} proved that the same upper bound holds for all $q\geq 14$.
We summarize his results as following.

\begin{thm}[F\"uredi, \cite{Fu83,Fur96}]\label{Thm:Furedi}
If $q\notin \{1,7,9,11,13\}$, then $\ex(q^2+q+1,C_4)\leq \frac{1}{2}q(q+1)^2$.
Hence for all prime powers $q\geq 14$, $\ex(q^2+q+1,C_4)=\frac{1}{2}q(q+1)^2$.
\end{thm}

\noindent
F\"uredi also proved that extremal graphs for $q\geq q_0$ must be orthogonal polarity graphs of order $q$ (unpublished, see \cite{Fur07}).
More recently, Firke, Kosek, Nash and Williford \cite{FKNW} proved $\ex(q^2+q,C_4)\leq \frac{1}{2}q(q+1)^2-q$ for all even integers $q$,
which implies that $\ex(q^2+q,C_4)=\frac{1}{2}q(q+1)^2-q$ for all $q=2^k$.
We remark that these results on $C_4$ belong to the category of rather rare exact results
for the notorious degenerate extremal graph problems (for a comprehensive survey, see F\"uredi-Simonovits \cite{FS}).
Very recently, we \cite{HMY20} proved a stability result on the Tur\'an number of 4-cycles,
which gives a structural description for those $C_4$-free graphs whose number of edges is close to the extremal graph.

\begin{thm}[\cite{HMY20}, Theorem 1.2]\label{Thm:stability 0}
Let $q\ge 10^9$ be even and $G$ be a $C_4$-free graph on $q^2+q+1$ vertices with at least $\frac12 q(q+1)^2-0.2q+1$ edges.
Then $G$ is a subgraph of a unique polarity graph of order $q$.
\end{thm}

The first main result of this paper extend the above result as follows.

\begin{thm}\label{Thm:stability}
Let $q$ be even and $G$ be a $C_4$-free graph on $q^2+q+1$ vertices with at least $\frac12 q(q+1)^2-\frac{1}{2}q+o(q)$ edges.
Then there exists a unique polarity graph of order $q$, which contains $G$ as a subgraph.
\end{thm}

In its proof we use many new technical ideas, some of which are from finite geometry.
We would like to note that the current method cannot break the ``$-\frac{1}{2}q$'' barrier due to the limit for bounding the maximum degree.
Before we discuss the applications of this result, we also like to point out that this stability result holds in a strong sense, namely, it only {\it adds} edges when turning graphs into desired configurations.
On the other hand, there are infinitely many examples showing that the same conclusion can not hold if the number of edges is lowered to $\frac12 q(q+1)^2-q+1$.
For $q=2^k$, there exists an orthogonal polarity graph $H$ of order $q$ with $\frac12 q(q+1)^2$ edges:
Choose non-adjacent $u, v\in V(H)$ with $d_H(u)=q+1$ and $d_H(v)=q$.
Let $G$ be obtained from $H$ by deleting all edges incident to $v$ and then adding a new edge $uv$.
Clearly $G$ is $C_4$-free and has $\frac12 q(q+1)^2-q+1$ edges.
However, $G$ can not be contained in any polarity graph $H'$ of order $q$,
as $d_G(u)=q+2$ and $\Delta(H')=q+1$.\footnote{In fact the number of edges can not be lowered to $\frac12 q(q+1)^2-q+2$ (this will be more involved to justify).}
For more discussion on this stability result, we direct readers to the concluding remarks.

As a first application of Theorem \ref{Thm:stability}, we can derive the following result on $\ex(n,C_4)$.
Let $\lambda(q)$ be the maximum number of edges in a polarity graph of order $q$ (if no such graphs exist, set $\lambda(q)=0$).

\begin{corollary}\label{coro:Turan1}
Let $q$ be even. If $\lambda(q)\geq \frac12 q(q+1)^2-\frac{1}{2}q+o(q)$, then $\ex(q^2+q+1,C_4)=\lambda(q)$,
where the equality holds only for polarity graphs of order $q$ with $\lambda(q)$ edges;
otherwise, $\ex(q^2+q+1,C_4)< \frac12 q(q+1)^2-\frac{1}{2}q+o(q)$.
In particular, $\ex(q^2+q+1,C_4)\leq \max \left\{\lambda(q), \frac12 q(q+1)^2-\frac12 q+o(q)\right\}$.
\end{corollary}

\noindent This corollary provides an improvement of F\"uredi's theorem for even integers $q$;
to be precise, it improves the upper bound by $\frac{1}{2}q-o(q)$ for at least infinite many even $q$ (we shall explain why later).
This and some other inference on $\ex(n,C_4)$ will be under further discussion in Section \ref{sec:Turan}.

A closely related extremal problem to Tur\'an numbers is ``the problem of supersaturated graphs" (quoted from \cite{ES84}).
It studies the following function: for a given graph $F$ and for integers $n, t\geq 1$,
\begin{equation*}\label{equ:h}
h_F(n,t)=\min\{\# F(G): |V(G)|=n, |E(G)|=\ex(n,F)+t\},
\end{equation*}
where $\# F(G)$ denotes the number of distinct copies of $F$ in a graph $G$.
A notable example is the study of the triangle $K_3$, started by Rademacher who proved that $h_{K_3}(n,1)=\lfloor n/2\rfloor$ in 1941.

Returning to our focus, throughout this paper we write $$h(n,t)=h_{C_4}(n,t) \mbox{ ~and ~} h(n)=h(n,1).$$
Analogously as Rademacher's result on the triangle, Erd\H{o}s and Simonovits conjectured that any $n$-vertex graph with $\ex(n,C_4)+1$ edges should contain many copies of $C_4$.
This problem repeatedly appeared in many papers of Erd\H{o}s.
A weak version (see \cite{Chung}, Conjecture 42) asserted that $h(n)\geq 2$ for large $n$, and another form (e.g. in \cite{E88,E90}) stated that $h(n)\geq c\sqrt{n}$ for some constant $c>0$.
The strongest version of this conjecture is the following.

\begin{conjecture}[Erd\H{o}s and Simonovits \cite{ES84}]\label{Conj:ES}
For integers $n$, $h(n)\geq \big(1+o(1)\big)\sqrt{n}$.
\end{conjecture}

\noindent As indicated in \cite{ES84}, if true, this bound will be sharp for infinitely many integers $n$.
We remark that a direct application of Theorem \ref{Thm:stability} already can show that $h(q^2+q+1)\geq \big(\frac12-o(1)\big)q$ for $q=2^k$.

We proved the following supersaturation result on $C_4$ in \cite{HMY20} using the stability result therein.

\begin{thm}[\cite{HMY20}, Theorem 1.3]\label{Thm:supersaturation-t=1}
Let $q\ge 10^{12}$ be even and let $G$ be a graph on $q^2+q+1$ vertices with $\frac{1}{2}q(q+1)^2+1$ edges.
Then either $G$ contains at least $2q-3$ copies of $C_4$,
or $G$ is obtained from an orthogonal polarity graph of order $q$ by adding a new edge.
In the latter case, $G$ contains $q-1, q$ or $q+1$ copies of $C_4$.
\end{thm}

As a corollary, this confirmed Conjecture \ref{Conj:ES} for an infinite sequence of integers $n$ as follows. 

\begin{corollary}[\cite{HMY20}]\label{coro:+1}
Let $q=2^k$ for $k\geq 40$. Then $h(q^2+q+1)=q-1$, where the graph achieves this equality
if and only if it is obtained from an orthogonal polarity graph of order $q$ by adding a new edge between (any) two vertices of degree $q$.
\end{corollary}

The second main result of this paper is to further characterize all extremal graphs for which achieve the $\ell${'th} least number of copies of $C_4$ for any fixed integer $\ell\geq 1$.

\begin{thm}\label{Thm:super-t=1-classify}
Let $q\gg \ell$ and $q$ be even.
Let $G$ be a graph on $q^2+q+1$ vertices with $\frac{1}{2}q(q+1)^2+1$ edges.
Then either $G$ has at least $(\ell+1)q-(\ell+1)^2$ copies of $C_4$,
or there exist some $s\in \{1,2,...,\ell\}$ and an orthogonal polarity graph $H$ of order $q$ such that $G$ can be obtained from $H$ by deleting or adding $2s-1$ edges.
In the latter case, the number of copies of $C_4$ in $G$ is between $sq-s^2$ and $sq+s^2$.
\end{thm}

\noindent This also indicates that the numbers of copies of $C_4$ among all such graphs are distributed sporadically (concentrated around $sq$ for small integers $s$).

For the general supersaturation problem of $C_4$, the function $h(n,t)$ is known to be $\Theta(t^4/n^4)$ when $t=\Omega(n^{3/2})$ for all $n$ (e.g. see \cite{ES84}).
We show in the following result that for an infinite sequence of integers $n$,
one can say rather accurately about the function  and in particular, one can determine the order of its magnitude for {\it every} positive integer $t$.

\begin{thm}\label{Thm:supersaturation}
The following statements hold for large $q=2^k$.
{\bf (A)} For every $1\leq t< q^{1/8}/30$, $h(q^2+q+1,t)=t(q-1)$,
where the equality holds for graphs $G$ if and only if $G$ is obtained from an orthogonal polarity graph of order $q$ by adding a matching of size $t$ among vertices of degree $q$.\\
{\bf (B)} For every $t\geq 1$, $h(q^2+q+1,t)\geq \big(\frac12+o(1)\big)tq$ and $h(q^2+q+1,t)=\Theta(tq+t^4/q^8)$.
\end{thm}

\noindent This follows by Theorems \ref{Thm:supersaturation-general-t} and \ref{Thm:supersaturation-general-general-t} which are stated in some more general settings.
We refer readers to Section \ref{sec:Super-general} for their precise statements.
As for general $n$, one also can determine the order of the magnitude of $h(n,t)$ when $t=\Omega(n^{3/2-\epsilon})$ for some $\epsilon\geq 0.2375$.

\begin{prop}\label{prop:1.2625}
Let $n$ be sufficiently large. If $t\geq 3n^{1.2625}$, then $h(n,t)=\Theta(t\sqrt{n}+t^4/n^4)$.
\end{prop}

The organization of this paper is as follows.
Section \ref{sec:preli} consists of preliminaries, where we give notations and collect some results.
In Section \ref{sec:intro-stability}, we outline the proof of Theorem \ref{Thm:stability}.
The full proof of Theorem \ref{Thm:stability} will be divided and completed in Sections \ref{sec:Delta=q+1}, \ref{sec:1-inters} and \ref{sec:polarity}.
In Section \ref{sec:Turan}, we prove Corollary \ref{coro:Turan1} and discuss other consequences on $\ex(n,C_4)$.
In Section \ref{sec:supert=1}, we prove Theorem \ref{Thm:super-t=1-classify}.
In Section \ref{sec:Super-general}, we prove Theorem \ref{Thm:supersaturation} and Proposition \ref{prop:1.2625} for the supersaturation problem of $C_4$.
In Section \ref{sec:concluding}, we discuss several problems in relation to the results here.
We would like to remark that though our results often are stated with parity condition,
many arguments in the proofs in fact work without any parity constraints.

\section{Preliminaries}\label{sec:preli}

\subsection{General notations}
We follow the notations of F\"uredi (e.g. \cite{Fur88}).
A {\it hypergraph} $\dH$ is an ordered pair $(V,\dE)$, where $V$ is a finite set consisting of {\it vertices} and $\dE$ is a collection of subsets (called {\it edges}) of $V$.
We use $e(\dH)$ to denote the number of edges in $\dH$.
For $x\in V$, the {\it degree} $d_{\dH}(x)$ of $x$ denotes the number of edges of $\dH$ containing $x$.
The {\it maximum degree} of $\dH$ is denoted by $\Delta(\dH)=\max\{d_\dH(x):x\in V\}$.
We say $\dH$ is {\it $k$-regular} if all vertices have degree $k$ and {\it $k$-uniform} if all edges have $k$ vertices.
A $k$-uniform hypergraph is also called a {\it $k$-graph} (and a {\it graph} if $k=2$).
We say $\dH$ is \emph{$1$-intersecting} if any two distinct edges of $\dH$ have exactly one common vertex.
The \emph{incidence matrix} of a hypergraph $\dH=(V,\dE)$ is an $|\dE|\times |V|$ matrix $\mathcal{M}(\dH)$
such that $\mathcal{M}(E,x)=1$ if $x\in E\in \dE$ and $0$ otherwise.

Let $G=(V, \dE)$ be a graph. Let $x\in V$ and $A\subseteq V$.
The {\it neighborhood} $N_G(x)$ of $x$ is the set of vertices $y\in V$ with $xy\in \dE$,
while the {\it closed neighborhood} $N_G[x]$ is defined by $N_G(x)\cup \{x\}$.
Let $N_A(x)=N_G(x)\cap A$.
Define $N_G(A)$ to be the set of vertices $u\in V\backslash A$ adjacent to some vertex in $A$ and $G[A]$ to be the subgraph of $G$ induced on $A$.
For a path $P$, its {\it length} $|P|$ denotes the number of edges it contains.
We say $P$ is a $k$-path if $|P|=k$.
For disjoint sets $A, B\subseteq V$, $e(A,B)$ denotes the number of edges $ab$ in $G$ with $a\in A$ and $b\in B$.
A set of edges is called {\it independent} if their endpoints are pairwise-disjoint.
For $u,v\in V$, we let $d_G(u,v)=|N_G(u)\cap N_G(v)|$.
We call $\{u,v\}$ an {\it uncovered} pair if $d_G(u,v)=0$ and a {\it covered} pair otherwise.
Let $UP$ be the set of uncovered pairs of $G$ and let $P_2$ be the set of all 2-paths in $G$.
The \emph{adjacency matrix} $\mathcal{A}(G)$ of $G$ is a $|V|\times |V|$ symmetric matrix such that $\mathcal{A}(x,y)=1$ if $xy\in \dE$ and $0$ otherwise.

Throughout this paper, the notation $\binom{x}{2}$ means the function $x(x-1)/2$ for all reals $x$.
For any positive integer $k$, we write $[k]$ as the set $\{1,2,...,k\}$.
For all above notations, we often drop the subscripts when they are clear from context.

\subsection{Projective planes}
A {\it finite projective plane of order $q$}, denoted by $PG(2,q)$, is a $(q+1)$-uniform $(q+1)$-regular $1$-intersecting hypergraph $\dH=(P,\dL)$ with $|P|=q^2+q+1$,
where $P$ consists of {\it points} and $\dL$ consists of {\it lines}.
It also follows that $|\dL|=q^2+q+1$ and any two points are contained in a unique line.
The existence of $PG(2,q)$ is well known for all prime powers $q$.
On the other hand, a major conjecture in this field asserts that the order $q$ of $PG(2,q)$ must be a prime power (known for $q\leq 11$ and still open for $q=12$).

A substantial body of our proofs will be involved with projective planes and 1-intersecting hypergraphs.
In preparation we now collect some related results.
The first two will play important roles for the constructive nature in our stability result (Theorem \ref{Thm:stability}).

\begin{thm}[\cite{Met91}]\label{Thm:Embedding}
Let $q\geq 3900$ and $\dH$ be a 1-intersecting $(q+1)$-hypergraph with $q^2+q+1$ vertices and more than $q^2-\frac{\sqrt{5}-1}{2}q+17\sqrt{q/5}$ edges.
Then $\dH$ can be embedded into a projective plane of order $q$.
\end{thm}

\begin{thm}[\cite{D83}]\label{Thm:Embedding unique}
Let $\dH$ be a 1-intersecting $(q+1)$-hypergraph with $q^2+q+1$ vertices and more than $q^2-q+1$ edges.
If $\dH$ can be embedded into a projective plane of order $q$, then this projective plane and the embedding both are unique.
\end{thm}

The following celebrated Bruck-Ryser theorem \cite{BR} gives a sufficient condition for the non-existence of projective planes of given order.

\begin{thm}[\cite{BR}]\label{Thm:BR}
If $q\equiv 1 \mbox{ or } 2\mod 4$ is an integer which cannot be expressed as a sum of two square numbers,
then there exist no projective planes of order $q$.
\end{thm}

We also need a useful lemma proved by F\"uredi (i.e., Lemma 3.7 in \cite{Fur88}).

\begin{lem}[\cite{Fur88}]\label{Lem:Furedi-symmetry}
Let $M=(m_{ij})$ be the incidence matrix of a projective plane of order $q$.
Suppose that $m_{ij}=m_{ji}$ whenever $1\leq i\leq q^2-q+3$ or $1\leq j\leq q^2-q+3$.
Then the whole matrix $M$ is symmetric.
\end{lem}

The coming lemma has been used in literatures (e.g. \cite{D83}),
which will serve as a handy tool for finding a large 1-intersecting hypergraph in the proof of Theorem \ref{Thm:stability}.
For completion, we give a proof.

\begin{lem}[e.g. \cite{D83}]\label{lem:enlarge R}
Let $\dH$ be a 1-intersecting $(q+1)$-hypergraph on vertex set $V$ with $|V|=q^2+q+1$.
Suppose that $\dF$ is a $(q+1)$-uniform hypergraph on the same vertex set $V$ such that $\dF\cap\dH=\emptyset$ and for any edge $f\in \dF$,
there exist $q$ edges $h_1,..., h_q\in \dH$ satisfying that $f\cup h_1\cup...\cup h_q=V$ and $|f\cap h_1\cap...\cap h_q|=1$.
Then $\dH\cup\dF$ is also $1$-intersecting.
\end{lem}

\begin{proof}
We first point out that to show this, it suffices to prove that for any $f\in \dF$, $\dH\cup \{f\}$ is $1$-intersecting.
This is because if we initially set $\dG=\dH$ and repeatedly operate by applying this for one edge $f\in \dF$ at a time and updating $\dG$ by $\dG\cup \{f\}$,
then in the end we would conclude that $\dG=\dH\cup\dF$ is $1$-intersecting.
Note that this indeed is valid as the conditions in the statement also hold for $\dG$ (instead of $\dH$) at any given time.

For the above desired statement, suppose on the contrary that there exist $f\in \dF$ and $h\in \dH$ such that $|h\cap f|=0$ or $|h\cap f|\geq 2$.
We know that there are $h_1,..., h_q\in \dH$ and $u\in V$ such that $f\cup h_1\cup...\cup h_q=V$ and $f\cap h_1\cap...\cap h_q=\{u\}$.
By the size of $V$, we also see that $f\backslash \{u\}, h_1\backslash \{u\},..., h_q\backslash \{u\}$ must form a partition of $V\backslash \{u\}$.
It is then clear that $h\notin \{f, h_1,\cdots, h_q\}$.
If $|h\cap f|=0$, then there must exist some $i\in [q]$ such that $|h\cap h_i|\geq 2$, a contradiction to that $\dH$ is $1$-intersecting.
Hence we may assume $|h\cap f|\geq 2$.
Suppose $u\in h$. Then $u\in h\cap h_i$ for all $i\in [q]$; since $\dH$ is 1-interesting, we conclude that $h\cap (h_1\cup ...\cup h_q)=\{u\}$ and thus $h=f$, a contradiction.
Now suppose $u\notin h$. Then $|h\cap (f\backslash \{u\})|\geq 2$ and thus there exists some $j\in [q]$ such that $|h\cap (h_j\backslash \{u\})|=0$,
which also shows that $|h\cap h_j|=0$, a contradiction.
We have completed the proof now.
\end{proof}

\subsection{Polarity graphs}

A \emph{polarity} $\pi$ of a projective plane $\dH=(P,\dL)$ is a bijection $\pi: P\cup\dL\to P\cup\dL$ such that
\begin{itemize}
\item $\pi^2$ is the identity function with $\pi: P\leftrightarrow \dL$, and
\item for any pair $(x,L)\in P\times\dL$, if $x\in L$ then $\pi(L)\in \pi(x)$.
\end{itemize}
For a projective plane $\dH=(P,\dL)$ of order $q$, where $P=\left\{x_i\right\}$ and $\mathcal{L}=\left\{L_i\right\}$,
consider a function $\phi: P\leftrightarrow \dL$ which maps $x_i\leftrightarrow L_{\sigma(i)}$ for some permutation $\sigma$ on $[q^2+q+1]$.
Let $\dM(\phi)$ be the incidence matrix of $\dH$, where the rows are listed in the order of $x_i$'s and the columns are listed in the order of $L_{\sigma(i)}$'s as $i$ increases.
It is worth pointing out that
\begin{equation}\label{equ:polairty}
\mbox{the function $\phi$ is a polarity $\Longleftrightarrow$ the incidence matrix $\dM(\phi)$ is symmetric.}
\end{equation}

Now let $\pi$ be a polarity of a projective plane $\dH=(P,\dL)$ of order $q$.
The \emph{polarity graph} $G(\pi)$ (of order $q$) is a simple graph on the vertex set $P$ such that $xy\in E(G(\pi))$ if and only if $x\in \pi(y)$.
A point $x\in P$ is called {\it absolute} (with respect to $\pi$) if $x\in \pi(x)$.
Let $a(\pi)$ denote the number of absolute points.
In \cite{Ba46}, Baer proved that
\begin{equation}\label{equ:a(pi)}
\mbox{there exists some integer $m_\pi\geq 0$ such that $a(\pi)=q+1+m_\pi\cdot\sqrt{q}$.}
\end{equation}
A polarity $\pi$ and its polarity graph $G(\pi)$ are called {\it orthogonal}, if $a(\pi)=q+1$ (i.e., $m_\pi=0$).
It is known that for any prime power $q$, there always exists an orthogonal polarity graph of order $q$.

Combining the above facts, it is easy to derive the following for polarity graphs.
\begin{prop}\label{prop:G(pi)}
Let $\pi$ be a polarity of order $q$.
Then the polarity graph $G(\pi)$ is a $C_4$-free graph on $q^2+q+1$ vertices with exactly $\frac12q(q+1)^2-\frac{m_\pi}2 \sqrt{q}$ edges such that every vertex has degree $q$ or $q+1$.
\end{prop}

The following lemma on polarity graphs is well-known (see Baer \cite{Ba45} for a proof).
\begin{lem}\label{lem:Baer}
Any two vertices of degree $q$ in a polarity graph of order $q$ are nonadjacent.
\end{lem}

The next lemma will be frequently used in the forthcoming proofs (see Lemma 2.4 in \cite{HMY20}).
\begin{lem}[\cite{HMY20}]\label{Lem:polarity+plus}
Let $G$ be a polarity graph of order $q$ with $uv\notin E(G)$.
Then $G\cup \{uv\}$ contains either $q-1, q$ or $q+1$ four-cycles,
any two of which share $uv$ as the unique common edge.
Moreover, $G\cup \{uv\}$ contains $q-1$ four-cycles if and only if both $u,v$ have degree $q$ in $G$.
\end{lem}

We also need a property on orthogonal polarity graphs of even order $q$ from \cite{Fur88}.

\begin{prop}\label{prop:even-q-w}
Let $q$ be even and $G$ be an orthogonal polarity graph of order $q$.
Then there exists a (unique) vertex $w$ of degree $q+1$ such that $N(w)$ consists of all vertices of degree $q$ in $G$.

\end{prop}

\subsection{$C_4$-free graphs}\label{subsec:C4-free}
We now give out notations arising from $C_4$-free graphs,
and along the way we also establish some statements for future use.

Throughout this subsection, let $G=(V,\dE)$ be a $C_4$-free graph on $n=q^2+q+1$ vertices.
For a vertex $v$, let $d_0(v)=|\{u\in V: \{u,v\}\in UP\}|$.

\begin{prop}\label{Prop:UP}
$|UP|=\frac{1}{2}\sum_{v\in V} d_0(v)$ and $|P_2|+|UP|=\binom{n}{2}$.
\end{prop}
\begin{proof}
The first equation follows by the definition.
Since $G$ is $C_4$-free, each 2-path corresponds to a unique covered pair.
Thus we have $|P_2|=\binom{n}{2}-|UP|$.
\end{proof}

We now introduce an important notation for our proofs.
For any $v\in V$, the \emph{deficiency} $f(v)$ of $v$ is defined by
$$f(v):=\max\{q+1-d(v),0\}.$$
The deficiency of a subset $A\subseteq V$ is $f(A)=\sum_{v\in A}f(v).$
Let $S_i=\{v\in V: d(v)=i\}$.

\begin{prop}\label{Prop:deficiency}
We have $f(N(v))\geq q$ for each $v\in S_{q+2}$, and if $\Delta(G)\leq q+2$ then $f(V)=(q+1)n-2e(G)+|S_{q+2}|$.
\end{prop}
\begin{proof}
Consider $v\in S_{q+2}$ with $N(v)=\{v_1,...,v_{q+2}\}$.
Since $G$ is $C_4$-free, $N(v_i)\backslash \{v\}$ are pairwise-disjoint,
which implies that $\sum_{v_i\in N(v)}(d(v_i)-1)\leq n-1$.
Thus we get
\begin{equation*}\label{Equ:f(Nv)}
f(N(v))\ge \sum_{v_i\in N(v)}(q+1-d(v_i))\ge (q+2)q-(n-1)=q.
\end{equation*}
Since $\Delta(G)\leq q+2$, it is straightforward to see that
$f(V)=\sum_{v\in V}\max\{q+1-d(v),0\}=(q+1)n-2e(G)+|S_{q+2}|.$
\end{proof}

Let $S=\{v\in V: d(v)\leq q\}$. We have the following lemma (see Corollary 5.2 in \cite{Fur88}).

\begin{lem}\label{Lem:neigbor-of-q-vertex}
If $q$ is even and $\Delta(G)=q+1$, then any vertex in $S_{q+1}$ has a neighbor in $S$ and moreover, $|S|\geq q+1$.
\end{lem}

\subsection{Others}
We need the following estimation on the distribution of prime numbers given in \cite{BHP}.

\begin{thm}[\cite{BHP}]\label{Thm:prime exists}
For sufficiently large $x>0$, the interval $[x-x^{0.525},x]$ contains prime numbers.
\end{thm}

At the end of this section, we give an easy-to-use lemma, which is often adopted in replace of standard Cauchy-Schwarz inequalities (see Lemma 2.6 in \cite{HMY20}).

\begin{lem}[\cite{HMY20}]\label{Lem:CS-inquality}
Let $a_1,...,a_m$ be nonnegative integers satisfying $\sum_{i=1}^{m}a_i\ge km+r$,
where $m,k,r$ are integers with $m,k>0$ and $r\geq -m$. Then we have
$\sum_{i=1}^{m}\binom{a_i}{2}\ge m\binom{k}{2}+rk.$
\end{lem}

\section{Proof outline of Theorem \ref{Thm:stability}}\label{sec:intro-stability}
In this section we discuss the proof of Theorem \ref{Thm:stability}.
For convenience, we restate Theorem \ref{Thm:stability} in the following thorough version.

\begin{thm}\label{Thm:stability-2}
For any $c\in (0,1)$, there exists some $q_c$ such that the following holds for even integers $q\geq q_c$.
If $G$ is a $C_4$-free graph on $q^2+q+1$ vertices with at least $\frac12 q(q+1)^2-\frac{c}{2}q$ edges,
then there exists a unique polarity graph of order $q$ containing $G$ as a subgraph.
\end{thm}

We now give a outline of the proof of Theorem \ref{Thm:stability-2}.
In a nutshell, it stems from the work of F\"uredi \cite{Fu83,Fur88,Fur96}.
Given a $C_4$-free graph $G$ on $q^2+q+1$ vertices with many edges, our goal is to construct a polarity graph of order $q$ containing $G$ as a subgraph.
This is achieved in the following three steps.

\begin{itemize}[leftmargin=14mm]
\item [Step 1.] We show that it suffices to consider for $\Delta(G)=q+1$.
\item [Step 2.] Let $\dR$ be the family of all subsets $N_G(x)$ where $x\in V(G)$ has degree $q+1$ and ``almost" all neighbors of $x$ have degree $q+1$.
Then we show that there exists a projective plane $\dH$ of order $q$ defined on $V(G)$ with $\dR\subseteq \dH$.
\item [Step 3.] We show that there exists a polarity $\pi$ of the above projective plane $\dH$ such that its polarity graph $G(\pi)$ contains $G$ as a subgraph.
\end{itemize}

To say more, Step 1 will be handled in Section \ref{sec:Delta=q+1},
where we reduce Theorem \ref{Thm:stability-2} to the following statement (with restriction $\Delta(G)=q+1$).

\begin{thm}\label{Thm:stability-max-degree}
For any $\epsilon\in (0,1)$, there exists some $q_\epsilon$ such that the following holds for even integers $q\geq q_\epsilon$.
If $G$ is a $C_4$-free graph on $q^2+q+1$ vertices with maximum degree $q+1$ and at least $\frac12 q(q+1)^2-\frac{\epsilon}{2}q$ edges,
then there exists a unique polarity graph of order $q$ containing $G$ as a subgraph.
\end{thm}

\noindent Our reduction shows that $\epsilon=c+o(1)$ holds for Theorem \ref{Thm:stability-max-degree} $\Longrightarrow$ $c$ holds for Theorem \ref{Thm:stability-2}.

For Theorem \ref{Thm:stability-max-degree}, we will divide its proofs into Sections \ref{sec:1-inters} and \ref{sec:polarity}.
In Section \ref{sec:1-inters}, we complete Step 2 by establishing Lemma \ref{Lem:finding-1-intersecting},
which asserts that there exists a 1-intersecting $(q+1)$-hypergraph $\dH$ containing $\dR$ and at least $q^2$ lines.
This indeed is enough to accomplish Step 2 as we can apply Theorem \ref{Thm:Embedding} to enlarge $\dH$ into a projective plane of order $q$ containing $\dR$.
The proof of Lemma \ref{Lem:finding-1-intersecting} is involved,
where the main technical difficulty lies in accurate analysis on the intricate relations between neighborhoods of vertices of degree $q$ or $q+1$.
Finally, we finish Step 3 in Section \ref{sec:polarity} and thus the proof of Theorem \ref{Thm:stability-max-degree}.
The arguments of this step will heavily rely on the properties of the family $\dR$.

\section{Reducing to $\Delta=q+1$}\label{sec:Delta=q+1}

In this section, as outlined earlier, we present a proof which reduces Theorem \ref{Thm:stability-2} to Theorem \ref{Thm:stability-max-degree}.

\medskip

\noindent {\bf Proof of Theorem \ref{Thm:stability-2} (Assuming Theorem \ref{Thm:stability-max-degree}).}
For any $c\in (0,1)$, we define $\epsilon$ to be any real in $(c,1)$ and choose $q_c$
such that $c+\frac{50}{\sqrt{q_c}}\leq \epsilon<1$ and $q_c\geq \max\big\{q_\epsilon, \frac{2500}{(\epsilon-c)^2}\big\}$, where $q_\epsilon$ is from Theorem \ref{Thm:stability-max-degree}.
Let $q\geq q_c$ be an even integer and let $G$ be a $C_4$-free graph on $n=q^2+q+1$ vertices such that
$$e(G)\geq \frac12 q(q+1)^2-\frac{c}2 q.$$
We will show that there exists a unique polarity graph of order $q$ containing $G$ as a subgraph.

Let $\Delta$ denote the maximum degree of $G$.
By the lower bound on $e(G)$, it is easy to see that $\Delta\geq q+1$.
If $\Delta=q+1$, since $e(G)\geq \frac12 q(q+1)^2-\frac{c}{2}q\geq \frac12 q(q+1)^2-\frac{\epsilon}{2}q$ and $q\geq q_c\geq q_\epsilon$,
our goal is accomplished by Theorem \ref{Thm:stability-max-degree}.
So we may assume that $q+2\leq\Delta\leq q^2+q$.

Let $V(G)=\{v_1,...,v_n\}$. We now process by showing a sequence of claims.

\begin{claim}\label{Cla:Delta=q+2}
$\Delta=q+2$.
\end{claim}
\begin{proof}
Suppose that $d(v_1)=\Delta \geq q+3$.
We now estimate the number $T$ of 2-paths in $G$ with none of its endpoints in $N(v_1)$.
Since any two vertices have at most one common neighbor and any two vertices in $N(v_i)$ are contained in a 2-path, we have
$$\binom{q^2+q+1-\Delta}{2}=\binom{n-\Delta}{2}\geq T\geq \sum_{i=2}^{n}\binom{|N(v_i)\setminus N(v_1)|}{2}.$$
Since $G$ is $C_4$-free, we see $|N(v_i)\setminus N(v_1)|=d(v_i)-d(v_i,v_1)\geq d(v_i)-1$ for $2\leq i\le n$.
As $q\geq q_c$ being sufficiently large, we have that
\begin{equation*}
\sum_{i=2}^{n}|N(v_i)\setminus N(v_1)|\ge 2e(G)-\Delta-(n-1)\geq (q^2+q)(q-1)+(q^2+2-\Delta).
\end{equation*}
As $q^2+2-\Delta\geq -(q^2+q)$, using Lemma \ref{Lem:CS-inquality} (with $m=n-1=q^2+q$), we have
$$\binom{q^2+q+1-\Delta}{2}\ge\sum_{i=2}^{n}\binom{|N(v_i)\setminus N(v_1)|}{2}\ge (q^2+q)\binom{q-1}{2}+(q-1)(q^2+2-\Delta).$$
After simplification, this is equivalent to that $g(\Delta):=\Delta^2-(2q^2+3)\Delta+(2q^3+5q^2-5q+4)\geq 0$, where $q+3\leq \Delta\leq q^2+q$.
It can be verified that $g(q+3)$ and $g(q^2+q)$ both are negative.
Since $g(\Delta)$ is quadratic, this shows that $g(\Delta)<0$ for all choices of $\Delta$, a contradiction.
This proves the claim.
\end{proof}

\begin{claim}\label{Cla:q+2 have common neighbor}
Any two vertices of degree $q+2$ have one common neighbor.
\end{claim}
\begin{proof}
Suppose for a contradiction that there exist $v_1,v_2\in S_{q+2}$ such that $N(v_1)\cap N(v_2)=\emptyset$.
Then for $3\le i\le n$ we have
$|N(v_i)\setminus (N(v_1)\cup N(v_2))|=|N(v_i)|-|N(v_i)\cap N(v_1)|-|N(v_i)\cap N(v_2)|\ge d(v_i)-2.$
Similarly as above, we estimate the number of 2-paths with none of its endpoints in $N(v_1)\cup N(v_2)$.
Using Jensen's inequality, we get that
\begin{equation*}
\begin{split}
\binom{n-2(q+2)}{2}
&\ge\sum_{i=3}^{n}\binom{|N(v_i)\setminus (N(v_1)\cup N(v_2))|}{2}
\ge(n-2)\binom{\frac{\sum_{i=3}^{n}(d(v_i)-2)}{n-2}}{2}\\
&= (n-2)\binom{\frac{2e(G)-2(q+2)-2(n-2)}{n-2}}{2}\geq (q^2+q-1)\binom{\frac{q^3-3q-cq-2}{q^2+q-1}}{2}
\end{split}
\end{equation*}
which is equivalent to
$$(q^2+q-1)(q^2-q-3)(q^2-q-4)\ge(q^3-3q-cq-2)(q^3-q^2-4q-cq-1).$$
After further simplification, we can derive that
$$2(c-1)q^4+(3-c)q^3+(11-7c-c^2)q^2-3(2+c)q-14\geq 0.$$
Since $0<c<1$, this inequality can not hold for large $q$, a contradiction.
\end{proof}

\begin{claim}\label{Cla:degree of common neighbor is small}
Any common neighbor of three vertices of degree $q+2$ has degree at most $q/2$.
\end{claim}
\begin{proof}
Suppose on the contrary that $v_1,v_2,v_3\in S_{q+2}$ have a common neighbor $v_4$ such that $d=d(v_4)>q/2$.
Let $A= N(v_4)\backslash\{v_1,v_2,v_3\}$, $B=V\backslash N[v_4]$ and $C=N(v_1)\cup N(v_2)\cup N(v_3)$.
So $|A|=d-3$ and $|B|=q^2+q-d$.
By similar discussion as before, we have that $|N(v)\setminus C|\geq d(v)-1$ for $v\in A$ and $|N(u)\setminus C|\geq d(u)-3$ for $u\in B\cup \{v_4\}$.
Therefore,
\begin{align*}
&\sum_{v\in A}(d(v)-1)+\sum_{v\in B}(d(v)-3)=\big(2e(G)-3(q+2)-d\big)-|A|-3|B|\\
\geq &(q(q+1)^2-cq)-3q^2-6q+d-3\geq (q^2+q-3)(q-2)-(q/2+9).
\end{align*}
Note that $-(q/2+9)>-(q^2+q-3)$.
Using Lemma \ref{Lem:CS-inquality}, if we estimate the number of 2-paths $T$ with none of its endpoints in $C$, then we can derive that
\begin{equation*}
\begin{split}
\binom{n-3(q+2)+2}{2}
&\geq T\geq \binom{d-3}{2}+\sum_{v\in A}\binom{d(v)-1}{2}+\sum_{v\in B}\binom{d(v)-3}{2}\\
&\geq\binom{q/2-3}{2}+(q^2+q-3)\binom{q-2}{2}-(q-2)(q/2+9).
\end{split}
\end{equation*}
This inequality is equivalent to $q^2-50q+72\leq 0$, which contradicts that $q$ is large.
\end{proof}

\begin{claim}\label{Cla:small degree vertex}
There are at most 7 vertices of degree at most $q/2$.
\end{claim}
\begin{proof}
Suppose on the contrary that there are at least 8 vertices of degree at most $q/2$.
Recall the notations in Subsection \ref{subsec:C4-free}.
For any $v\in S_{q+1-k}$ for $k\in [q]$, we have
\begin{align*}
d_0(v)&=(n-1)-\sum_{u\in N(v)}(d(u)-1)=(q^2+2q+1-k)-\sum_{u\in N(v)}d(u)\\
&\ge(q^2+2q+1-k)-(q+2)(q+1-k)=(k-1)(q+1),
\end{align*}
where the inequality holds as $\Delta=q+2$.
By Proposition \ref{Prop:UP}, we have
\begin{equation*}\label{Equ:UP}
\binom{n}{2}-|P_2|=|UP|=\frac{1}{2}\sum_{v\in V}d_0(v)\ge\frac{q+1}{2}\sum\limits_{k=1}^{q}(k-1)|S_{q+1-k}|\ge\frac{q+1}{2}\cdot \frac{q}{2}\cdot 8=2q^2+2q.
\end{equation*}
We also have $|P_2|=\sum_{v\in V}\binom{d(v)}{2}$ and $\sum_{v\in V}d(v)\geq q(q+1)^2-cq=(q^2+q+1)(q+1)-(q+cq+1)$.
Since $-(q+cq+1)>-(q^2+q+1)$, by Lemma \ref{Lem:CS-inquality} we get
\begin{align*}
|P_2|\geq (q^2+q+1)\binom{q+1}{2}-(q+1)(q+cq+1)=\binom{n}{2}-(q+1)(q+cq+1).
\end{align*}
Combining with the above two inequalities, we derive that
$$ 2q^2+2q-(q+1)(q+cq+1)=(q+1)((1-c)q-1)\leq 0.$$
Again, this contradicts the fact that $q$ is large and thus completes the proof of Claim \ref{Cla:small degree vertex}.
\end{proof}

\begin{claim}
$|S_{q+2}|\le 22\sqrt{q}$.
\end{claim}
\begin{proof}
Let $S'=\{v: d(v)\le q/2\}$ and $S_{q+2}'=S_{q+2}\backslash N(S')$.
By Claim \ref{Cla:small degree vertex} we have $|S'|\le 7$,
so $|S_{q+2}\cap N(S')|\leq 7q/2$ and $|S_{q+2}'|\ge |S_{q+2}|-7q/2$.
Recall $S=\{v: d(v)\le q\}$.

We now define a weight function $w$ on the edges $xy$ with $x\in S_{q+2}'$ and $y\in S$ by assigning $w(xy)$ to be the deficiency $f(y)$.
We consider the total weight $W$ of these edges.
On the one hand, by Proposition \ref{Prop:deficiency}, every $S_{q+2}'$ vertex contributes at least $q$ and
$$f(S)=f(V)=(q+1)n-2e(G)+|S_{q+2}|\leq q+cq+1+|S_{q+2}|.$$
On the other hand, by Claim \ref{Cla:degree of common neighbor is small},
every vertex in $S$ is adjacent to at most two vertices in $S_{q+2}'$ and thus contributes at most twice of its deficiency.
Putting these together, we have
\begin{equation}\label{Equ:fV-2}
q(|S_{q+2}|-7q/2)\leq q|S_{q+2}'|\leq W\leq 2f(S)\leq 2(q+cq+1+|S_{q+2}|),
\end{equation}
from which we derive that $|S_{q+2}|\leq \frac{3.5q^2+2(1+c)q+2}{q-2}\leq 4q$.
For any $v\in S'$, let $N(v)\cap S_{q+2}=\{u_1,...,u_t\}$.
Since $N(u_i)\backslash \{v\}$ are disjoint for all $i\in [t]$, by Proposition \ref{Prop:deficiency} it follows that
\begin{align*}
f(V)\geq f(v)+\sum_{i=1}^{t}f(N(u_i)\backslash \{v\})\geq f(v)+t(q-f(v))=(d(v)-1)(t-1)+q.
\end{align*}
Since $d(v)\geq t$ and $|S_{q+2}|\le 4q$, by \eqref{Equ:fV-2} we have
$$(t-1)^2+q\leq f(V)\leq q+cq+1+|S_{q+2}|\le 6q+1,$$ implying that $|N(v)\cap S_{q+2}|=t\leq \sqrt{5q+1}+1\leq 3\sqrt{q}$ for any $v\in S'$.

We improve the estimation of $|S_{q+2}\cap N(S')|\leq 21\sqrt{q}$ and can run the above procedure again.
Since $|S_{q+2}'|\geq |S_{q+2}|-21\sqrt{q}$, we have
$$q(|S_{q+2}|-21\sqrt{q})\le q|S_{q+2}'|\leq W\leq 2f(S)\leq 2(q+cq+1+|S_{q+2}|).$$
This shows that $|S_{q+2}|\leq \frac{21q\sqrt{q}+2(1+c)q+2}{q-2}\leq 22\sqrt{q}$.
\end{proof}

We are ready to complete the proof.
Since $|S_{q+2}|\le 22\sqrt{q}$, we can delete at most $22\sqrt{q}$ edges from $G$ to get a subgraph $G'$ with maximum degree $q+1$.
By the choice of $\epsilon$, we have
\begin{equation}\label{equ:e(G')}
e(G')\geq e(G)-22\sqrt{q}\geq \frac{1}{2}q(q+1)^2-\frac{c}{2}q-22\sqrt{q}\geq \frac{1}{2}q(q+1)^2-\frac{\epsilon}{2}q.
\end{equation}
By Theorem \ref{Thm:stability-max-degree},
there exists a unique polarity graph $H$ containing $G'$ as a subgraph.
Let $e_1,...,e_t$ be the edges deleted from $G$, where $t\leq 22\sqrt{q}$.
We may assume that $t\geq 1$, as otherwise $G=G'$ is a subgraph of $H$.
So we have $e_1=xy\notin E(H)$.
By Lemma \ref{Lem:polarity+plus}, $H\cup \{e_1\}$ contains at least $q-1$ copies of $C_4$,
all of which contain $e_1$ and are edge-disjoint otherwise.

Consider $G'\cup \{e_1\}$, which is a subgraph of $G$ and thus is $C_4$-free.
So any of these $C_4$'s found in $H\cup \{e_1\}$ must have an edge not in $G'\cup \{e_1\}$, which are pairwise-distinct.
This shows that $e(G')\leq e(H)-(q-1)\leq \frac{1}{2}q(q+1)^2-(q-1)$,
which contradicts \eqref{equ:e(G')}.
This finishes the proof of Theorem \ref{Thm:stability-2}.
\QED

\bigskip

We point out that the reduction in this section works for {\it all} integers $q$, not necessarily for even $q$.

\section{Finding a large 1-intersecting hypergraph}\label{sec:1-inters}
We prove Theorem \ref{Thm:stability-max-degree} in the following two sections.
The goal of this section is to construct a 1-intersecting $(q+1)$-hypergraph,
which represents the $C_4$-free graph considered in Theorem \ref{Thm:stability-max-degree} (see Lemma \ref{Lem:finding-1-intersecting} below).

Let $\epsilon$ and $G$ be from Theorem \ref{Thm:stability-max-degree} throughout this section.
Namely, $\epsilon\in (0,1)$ is a fixed constant and $G=(V,\dE)$ is a $C_4$-free graph on $n=q^2+q+1$ vertices with maximum degree $q+1$ and
\begin{equation}\label{equ:e(G)+1}
e(G)\geq \frac{1}{2}q(q+1)^2-\frac{\epsilon}{2}q,
\end{equation}
where $q$ is an even integer.
Moreover, we choose $q$ and $\delta=\delta(\epsilon)$ such that
\begin{equation}\label{equ:choices-of-q&delta}
\frac{1}{q}\ll \delta\ll 1-\epsilon.
\end{equation}

The following notations will play essential roles in the proofs of the coming two sections.

\begin{dfn}\label{dfn:def-A-B-R}
Let $B=\{x\in V: |N(x)\cap S|\geq \delta q\}$ and $A=S_{q+1}\backslash B$.
Let $\dR=\{N(x):x\in A\}$. We call any subset in $V$ of size $q+1$ a {\it line}.
\end{dfn}

Now we are able to state the main result of this section.

\begin{lem}\label{Lem:finding-1-intersecting}
There is a 1-intersecting $(q+1)$-hypergraph on the vertex set $V$, which contains $\dR$ and at least $q^2$ lines.
\end{lem}

The full proof of this is involved, which we break into two subsections.
However, Subsection \ref{subsec:5.1} is concise and about 2-page long,
which indicates that Lemma \ref{Lem:finding-1-intersecting} holds for any real
$\epsilon\in (0,1/2)$.\footnote{Note that Claim \ref{Cla:R+L} yields a 1-intersecting $(q+1)$-hypergraph containing $\dR$ and at least $q^2+(1-2\epsilon)q+O(1)$ lines.}
Subsection \ref{subsec:5.2} requires rather technical efforts to show Lemma \ref{Lem:finding-1-intersecting} for all $\epsilon \in (0,1)$.
We like to mention that Subsection \ref{subsec:5.2} is not necessary for the stability with a weaker bound such as in \cite{HMY20}.

\subsection{An initiatory bound}\label{subsec:5.1}
The proof of Lemma \ref{Lem:finding-1-intersecting} will process by showing a sequence of claims.
Before that, we first collect some basic properties on $G$.
By Proposition \ref{Prop:deficiency} and Lemma \ref{Lem:neigbor-of-q-vertex}, we have
\begin{equation}\label{equ:new fV}
q+1\leq |S|\leq \sum_{i=0}^q (i+1)|S_{q-i}|= f(V)=(q+1)n-2e(G)\leq q+\epsilon q+1
\end{equation}
and thus
\begin{equation}\label{equ:|Sq+1|}
q^2-\epsilon q\leq |S_{q+1}|\leq q^2.
\end{equation}
For any $T\subseteq S$,
it holds that $q+\epsilon q+1\geq f(V)\geq f(T)+(|S|-|T|)\geq f(T)+(q+1-|T|)$.
This implies that $$f(T)\leq |T|+\epsilon q \mbox{ ~for any } T\subseteq S$$
and in particular, one can derive that
\begin{equation}\label{equ:dx+dy}
d(x)\geq (1-\epsilon)q \mbox{~ and~ } d(x)+d(y)\geq (2-\epsilon)q \mbox{~ for any } x,y\in V.
\end{equation}

In the same way as we did in \cite{HMY20} (see Claim 3.1 in \cite{HMY20}), we can prove the first claim by deriving some bounds on the sizes of $A$ and $B$ (we omit the proof here, since they are similar).

\begin{claim}\label{Cla:|A|-|B|}
$|B|\leq \frac{2}{\delta}$ and $|A|\geq |S_{q+1}|-|B|\geq q^2-\epsilon q-\frac{2}{\delta}.$
\end{claim}

Next we investigate properties on some special vertices of degree $q+1$, defined as following.
We remark that by Lemma \ref{Lem:neigbor-of-q-vertex}, any vertex in $S_{q+1}$ is adjacent to at least one vertex in $S$.

\begin{dfn}\label{dfn:V1}
A vertex $v\in V$ has {\bf property 1}, if $v\in S_{q+1}$ satisfies that $|N(v)\cap S_{q+1}|=q$ and $|N(v)\cap S_q|=1$.
Let $V_1$ denote the set of all vertices of property 1 in $G$.
\end{dfn}

We also can derive the following analogue of the size of $|V_1|$ (see the beginning of the proof of Claim 3.2 in \cite{HMY20}).

\begin{claim}\label{Cla:the size of V1}
$|V_1|\ge(1-\epsilon)q^2-(1+2\epsilon)q.$
\end{claim}

Suppose $v\in V_1$ has $N(v)=\{v_1,...,v_{q+1}\}$.
Let $N_{i}=N(v_{i})\setminus N[v]$ for $i\in [q+1].$
The next claim describes the structure of the neighborhood of a vertex in $V_1$ (see Proposition 5.4 in \cite{Fur88}).

\begin{claim}\label{Cla:V1-structure}
For $v\in V_1$, the sets $N_1,..., N_{q+1}$ form a partition of $V\backslash N[v]$,
and $G[N(v)]$ consists of a matching of size $\frac{q}{2}$ plus an isolated vertex of degree $q$.
\end{claim}

The following is key for constructing a $(q+1)$-uniform 1-interesting hypergraph (see (7) in \cite{HMY20}).

\begin{claim}\label{Cla:1-intersecting-tool}
Suppose $v\in V_1$ has $N(v)=\{v_1,...,v_{q+1}\}$.
If $u\in S_{q+1}\setminus N[v]$ is adjacent to $S_{q+1}\cap N(v)$,
then $|N(u)\cap N(v_i)|=1$ for all $i\in [q+1]$.
\end{claim}

We then show that the neighborhood of any vertex in $A$ contains many vertices of property 1.
Following the definitions in \cite{HMY20},
for any $x\in A$ we define
$S_x=N(x)\cap S$ and $S^*_x=S_x\cup (N(S_x)\cap N(x))$.
Since $x\in A$, we have $|S_x|\le \delta q$.
Every vertex in $S_x$ has at most one neighbor in $N(x)$, so $|S^*_x\backslash S_x|\leq |S_x|$ and thus $|S^*_x|\leq 2|S_x|\leq 2\delta q$. Then we can get the following claim (see the proof of Claim 3.2 in \cite{HMY20}).

\begin{claim}\label{Cla:many-prop1-for-A-vtx}
For $x\in A$, there are at least $(1-\epsilon-3\delta)q+1$ vertices of property 1 in $N(x)\backslash S^*_x$.
\end{claim}

Finally, combining the above claims, the selection of $\delta$ and the analysis at the ending of the proof of Claim 3.2 in \cite{HMY20}, we can conclude the following claim.

\begin{claim}\label{Cla:R}
$\dR$ is a 1-intersecting $(q+1)$-hypergraph with $|\dR|\geq q^2-\epsilon q-2/\delta$.
\end{claim}

\begin{dfn}\label{dfn:L-family}
Let $V_1^*=V_1\backslash N(B)$. For $v\in V_1^*$, let $v'$ be the unique vertex in $N(v)$ of degree $q$.
Let $\mathcal{L}=\{N[v']: v\in V_1^*\}$.
\end{dfn}

It is clear that $\mathcal{R}\cap \mathcal{L}=\emptyset$.

\begin{claim}\label{Cla:L-property}
For any $L\in \mathcal{L}$, there exist $R_1,...,R_q\in \dR$ such that $L\cup R_1\cup...\cup R_q=V$ and $|L\cap R_1\cap...\cap R_q|=1$.
Moreover, $|\dL|\geq (1-\epsilon)q-(2+2\epsilon+2/\delta)$.
\end{claim}

\begin{proof}
Given $L\in \mathcal{L}$, there exists a vertex $v\in V_1^*$ with $N(v)=\{v_1,...,v_{q+1}\}$ and $L=N[v_{q+1}]$.
Since $N(v)\cap B=\emptyset$, we see that $v_1,...,v_q\in A$ and thus $N(v_1),...,N(v_q)\in \dR$.
By Claim \ref{Cla:V1-structure}, it implies that $L\cup N(v_1)\cup...\cup N(v_q)=V$ and $L\cap N(v_1)\cap ...\cap N(v_q)=\{v\}$, as desired.
It remains to show the lower bound of $|\dL|$.
By Claims \ref{Cla:|A|-|B|} and \ref{Cla:the size of V1}, we have $|V_1^*|\geq |V_1|-|N(B)|\geq (1-\epsilon)q^2-(1+2\epsilon)q-2(q+1)/\delta$.
It is obvious that each $L\in \mathcal{L}$ corresponds to a unique vertex $u\in S_q\backslash B$ with $L=N[u]$ (as otherwise it will force a $C_4$),
however such a vertex $u$ may be adjacent to at most $q$ vertices in $V_1\backslash N(B)$.
Thus we have that
$|\dL|\geq \frac{|V_1^*|}{q}\geq \frac{1}{q}\left((1-\epsilon)q^2-(1+2\epsilon)q-2(q+1)/\delta\right) \geq (1-\epsilon)q-(2+2\epsilon+2/\delta).$
\end{proof}

Using the above claims, we now can construct a 1-intersecting $(q+1)$-hypergraph with many edges.

\begin{claim}\label{Cla:R+L}
$\dR\cup \dL$ is a 1-intersecting $(q+1)$-hypergraph based on $G$ with $|\dR\cup \dL|\geq q^2+(1-2\epsilon)q-(2+2\epsilon+4/\delta)$.
\end{claim}
\begin{proof}
In view of Lemma \ref{lem:enlarge R}, we see from Claims \ref{Cla:R} and \ref{Cla:L-property} that $\dR\cup \dL$ is a 1-intersecting $(q+1)$-hypergraph.
Since $\mathcal{R}\cap \mathcal{L}=\emptyset$, we have
$|\dR\cup \dL|\geq (q^2-\epsilon q-2/\delta)+(1-\epsilon)q-(2+2\epsilon+2/\delta)=q^2+(1-2\epsilon)q-(2+2\epsilon+4/\delta),$
completing the proof.
\end{proof}

\subsection{The completion of the proof of Lemma \ref{Lem:finding-1-intersecting}}\label{subsec:5.2}
Following the sequence of previous claims, we now continue and complete the proof of Lemma \ref{Lem:finding-1-intersecting}.
By \eqref{equ:|Sq+1|}, we write
\begin{equation}\label{equ:Sq+1=alpha}
|S_{q+1}|=q^2-\alpha q \mbox{~ for some constant } 0\leq \alpha\leq \epsilon.
\end{equation}
Using $|V|=\sum_{i=0}^{q+1}|S_i|$ and $2e(G)=\sum_{i=0}^{q+1}i|S_i|$,
by \eqref{equ:e(G)+1} and \eqref{equ:Sq+1=alpha} we can conclude that
\begin{equation}\label{equ sq}
|S_q|\geq 2e(G)-(q-1)|V|-2|S_{q+1}|\geq (2\alpha+1-\epsilon)q+1.
\end{equation}

\begin{dfn}\label{dfn:V2}
We say a vertex $v$ has {\bf property 2}, if $v\in S_{q+1}\backslash N(B)$ satisfies $|N(v) \cap S_{q+1}|=q-1$ and $|N(v) \cap S_q|=2$.
Let $V_2$ denote the set of all vertices of property 2 in $G$.
\end{dfn}

\begin{claim}\label{Cla:many prop 2}
Any $u\in S_q\backslash (B\cup N(V_1^*))$ has at least $(1-\epsilon-3\delta)q-(1+6/\delta)$ neighbors in $V_2$.
\end{claim}

\begin{proof}
Let $S_u=N(u)\cap S$ and $S_u^*=S_u\cup (N(S_u)\cap N(u))$.
Since $u\notin B$ and  every vertex in $S_u$ has at most one neighbor in $N(u)$,
we have $|S_u^*\backslash S_u|\leq |S_u|\leq \delta q$ and $|S_u^*|\leq 2\delta q$.

Write $N(u)=\{u_1,...,u_q\}$ and $N_{i}=N(u_{i})\setminus N[u]$ for each $i\in [q]$.
Since $G$ is $C_4$-free, all the sets $N_i$ are disjoint.
Note that $N(u)\cap V_1^*=\emptyset$.
So $N(u)\cap V_1\subseteq N(u)\cap N(B)$.
Since each vertex in $B$ has at most one neighbor in $N(u)$,
we have $|N(u)\cap V_1|\leq |N(u)\cap N(B)|\leq |B|\leq 2/\delta$ and thus $|N(u)\backslash V_1|\geq q-2/\delta$.
Each $u_i\in N(u)\backslash (V_1\cup S_u^*)$ has degree $q+1$ and at least two neighbors in $S$,
all of which (except $u$) are in $V\backslash N[u]$.
Hence for such $u_i\in N(u)\backslash (V_1\cup S_u^*)$, the deficiency $f(N_i)$ is at least 1 with equality if and only if $|N(u_i)\cap S_{q+1}|=q-1$ and $|N(u_i)\cap S_q|=2$.
Let $M$ be the set of vertices $u_i\in N(u)\backslash (V_1\cup S_u^*)$ with $f(N_i)=1$.
By counting the deficiency, we have
$$|M|+2(|N(u)\backslash V_1|-|S_u^*|-|M|)+|S_u|\leq \sum_{i=1}^q f(N_i\cup \{u_i\})\leq f(V)\leq q+\epsilon q+1.$$
Using $|N(u)\backslash V_1|\geq q-2/\delta$, we can derive that $|M|\geq (1-\epsilon-3\delta)q-(1+4/\delta)$.
Notice that $M\backslash N(B)\subseteq V_2$.
Since $M\subseteq N(u)$, each vertex in $B$ has at most one neighbor in $M$ and thus we have at least
$|M\backslash N(B)|\geq |M|-|B|\geq (1-\epsilon-3\delta)q-(1+6/\delta)$ vertices in $N(u)\cap V_2$.
\end{proof}

\begin{claim}\label{Cla:V2-structure}
For $v\in V_2$, $G[N(v)]$ consists of a matching of size $\frac{q}{2}$ plus an isolated vertex of degree $q$.
\end{claim}
\begin{proof}
We observe that every $v\in V_2$ is in $A$ and every $x\in N(v)\cap S_{q+1}$ is also in $A$.
By Claim \ref{Cla:R}, $|N(v)\cap N(x)|=1$.
Since $v$ has $q-1$ neighbors of degree $q+1$ and $q$ is even,
it is easy to see that $G[N(v)]$ must contain exactly $\frac{q}{2}$ edges which form a matching and moreover, the only isolated vertex in $G[N(v)]$ has degree $q$.
\end{proof}

In view of Claim \ref{Cla:V2-structure}, we now give some notations in relation to $V_2$ for later use.

\begin{dfn}\label{dfn:more-on-V2}
For $v\in V_2$, we write $N(v)=\{v_{1},...,v_{q+1}\}$ such that $d(v_1)=d(v_2)=q$ and the edge set of $G[N(v)]$ is $\{v_{2i}v_{2i+1}: 1\leq i\leq \frac{q}{2}\}$.
Let $N_i=N(v_i)\backslash N[v]$ for each $i$, and let the unique vertex not contained in $\cup_{1\leq i\leq q+1}N[v_i]$ be $v^\star$.
\end{dfn}

Note that the sets $N[v], N_1,...,N_{q+1}$ and $\{v^\star\}$ form a partition of $V$.
We recall that $\dR=\{N(x): x\in A\}$ is $1$-intersecting and subsets in $V$ of size $q+1$ are called {\it lines}.

\begin{dfn}\label{dfn:extendable-V2}
For $v\in V_2$, we say $v_1$ and $v_2$ are the {\it type-\1 vertex} and {\it type-\2 vertex} for $v$, respectively.
We also say that $L_{\1}^v=N[v_1]$ is the {\it type-\1 line} for $v$ and $L_{\2}^v=N(v_2)\cup \{v^\star\}$ is the {\it type-\2 line} for $v$.\footnote{We will also say that, for instance, $v_1$ is the vertex of type-\1 for $v$ and $L_{\2}^v$ is the line of type-\2 for $v$.}
Furthermore, we say $v\in V_2$ is {\bf extendable} if $\dR\cup \{L_{\1}^v, L_{\2}^v\}$ is $1$-intersecting; otherwise, we call it {\bf non-extendable}.
\end{dfn}

\begin{claim}\label{Cal:RcapV2}
For $v\in V_2$, we have $v_3,...,v_{q+1}\in A$, $L_{\1}^v\cup L_{\2}^v\cup N(v_3)\cup...\cup N(v_{q+1})=V$ and $L_{\1}^v\cap L_{\2}^v\cap N(v_3)\cap...\cap N(v_{q+1})=\{v\}$.
Moreover for each $N\in \dR$, $|N\cap L_{\1}^v|+|N\cap L_{\2}^v|=2$.
\end{claim}
\begin{proof}
It is clear that the first conclusion follows by definition.
Consider any $v\in V_2$ and $N\in \dR$.
Note that $v_i\in A$ for $3\leq i\leq q+1$.
If $N=N(v_i)$ for some $3\leq i\leq q+1$, then the second conclusion is clear.
Otherwise, as $\dR$ is 1-intersecting, we see that $|N\cap N(v_i)|=1$ for $3\leq i\leq q+1$.
This also infers that $|N\cap (L_\1^v\cup L_\2^v)|=2$, completing the proof.
\end{proof}

Therefore, if $v\in V_2$ is non-extendable, then there must be a line $N\in\dR$ with $(|N\cap L_{\1}^v|,|N\cap L_{\2}^v|)\in \{(0,2),(2,0)\}$.
This motivates the following definitions.

\begin{dfn}\label{dfn:good-bad-vertices}
If $v\in V_2$ is extendable, we say both vertices $v_1, v_2$ and both lines $L_\1^v, L_\2^v$ are {\bf good} (of type-\1, type-\2 respectively) for $v$.
Suppose $v\in V_2$ is non-extendable.
If there exists $N\in \dR$ with $|N\cap L_{\1}^v|=2$,
then the type-\1 vertex $v_1$ is called {\bf bad} for $v$.
If there exists $N'\in \dR$ with $|N'\cap L_{\2}^v|=2$,
then the type-\2 vertex $v_2$ is called {\bf bad} for $v$.
\end{dfn}

We point out that any vertex of type-\1 or -\2 for some vertex in $V_2$ belongs to the set $S_q\backslash B$.
Recall that $V_1^*=V_1\backslash N(B)$ and $\dL=\{N[u]: u\in S_q\cap N(V_1^*)\}$ from Definition \ref{dfn:L-family}.

\begin{claim}\label{Cla:N(V1^*)}
Let $u$ be a vertex of type-\1 or -\2 for some $v\in V_2$.
If $u\in N(V_1^*)$, then $N[u]\in \dL$ and $u$ is a good vertex of type-\1 for $v$.
\end{claim}

\begin{proof}
Suppose that $u$ is adjacent to $w\in V_1\backslash N(B)$.
Then $N[u]\in \dL$ and by Claim \ref{Cla:R+L}, $\dR\cup \{N[u]\}$ is 1-intersecting.
If $u$ is a type-\2 vertex for $v$, then there exists some $u'\in N(v)\cap A$ adjacent to $u$.
So $N(u')\in \dR$, but $\{u,v\}\subseteq N(u')\cap N[u]$, a contradiction.
Hence $u$ is type-\1 for $v$. In particular, $L_\1^v=N[u]$.
By Claim \ref{Cal:RcapV2}, applying Lemma \ref{lem:enlarge R} with $\dH=\dR\cup \{N[u]\}$ and $\dF=\{L_\2^v\}$,
we derive that $\dR\cup \{L_\1^v, L_\2^v\}$ is 1-intersecting.
This shows that $v$ is extendable, finishing the proof.
\end{proof}

We now investigate more properties for general $u\in S_q\backslash B$.
It seems possible that $u$ can be good or bad for different vertices in $V_2\cap N(u)$.
Nevertheless, we will show in the following claims that there are strong restrictions one can say for the goodness/badness.

\begin{claim}\label{Cla:non-extendable-V2}
If $v\in V_2$ is non-extendable, then there always exists $N\in \dR$ with $|N\cap L_{\1}^v|=2$.
\end{claim}
\begin{proof}
Let $N_i=N(v_i)\backslash N[v]$ for $i\in [q+1]$.
Suppose not. Then by Claim \ref{Cal:RcapV2}, there exists some $N(c)\in \dR$ with $|N(c)\cap L_{\2}^v|=2$.
It is clear that $v^\star\in N(c)$.
Since $c\in A$, by Claim \ref{Cla:many-prop1-for-A-vtx} there are at least $(1-\epsilon-3\delta)q+1$ many neighbors of $c$ in $V_1\backslash S_c^*$.
As $G$ is $C_4$-free, we can find at least $(1-\epsilon-3\delta)q-2/\delta+1$ vertices in $(V_1\cap N(c))\backslash N(B)$.

Now we take $b\in(V_1\cap N(c))\backslash\big(N(B)\cup N(v_1)\cup N(v_2)\cup N(v_3)\cup N(v)\cup N(v^\star)\big)$.
Suppose $N(b)\cap N(v)=\{v_i\}$ (by Claim \ref{Cla:R} they intersect).
By the choice of $b$, we may write $N(b)=\{b_1,...,b_{q+1}\}$ with $b_1=v_i$, where $d(b_1)=q+1$ and $d(b_{q+1})=q$.
Since $e(N(b), N_2\cup \{v_3\})\leq | N_2\cup \{v_3\}|=q-1$,
there are at least two vertices in $N(b)$ with no neighbor in $N_2 \cup \{v_{3}\}$.
So there exists $a\in N(b)\backslash \{v_i\}=N(b)\backslash N(v)$ with $N(a)\cap (N_2 \cup \{v_{3}\})=\emptyset$.

We assert that $a\in N_1\cup ...\cup N_{q+1}$ and $N(a)\cap L_\2^v=\emptyset$.
First, we see that $av\notin E(G)$ (as otherwise, $a\in N(v)\cap N(b)=\{v_i\})$.
It then suffices to show $v^\star\notin N(a)$.
To see this, suppose $v^\star\in N(a)$ and then one can find a $C_4$, namely $abcv^\star a$ in $G$, a contradiction.

We also assert that it has to be $d(a)=q$.
It is clear that $d(a)\in \{q, q+1\}$.
Suppose for a contradiction that $d(a)=q+1$.
Since $b\in V_1\backslash N(B)$, we have $a\in A$ and thus $N(a)\in \dR$.
Since $v_3,...,v_{q+1}\in A$,
by Claim \ref{Cla:R} we have $|N(a)\cap N(v_\ell)|=1$ for $3\leq \ell\leq q+1$.
Then by Claim \ref{Cal:RcapV2}, $|N(a)\cap L_\1^v|=2$, contradicting our assumption.

We now further show that $a\in N_1$.
Assume for a contradiction that $av_1\notin E(G)$.
As $av_{3}\notin E(G)$, we have $a\in N_2\cup N_4\cup ...\cup N_{q+1}$.
Suppose $a\in N_2$.
By Claim \ref{Cla:V1-structure}, we see that $\{b_1,...,b_q\}$ induces a matching in $G[N(b)]$;
then at least two of $\{b_1,...,b_q\}$ have no neighbor in $N_2\cup \{v_3\}$.
Hence we can choose $a'\in \{b_1,...,b_q\}\backslash \{v_i\}$ with $N(a')\cap (N_2 \cup \{v_{3}\})=\emptyset$.
But such $a'$ has degree $q+1$, contradicting the previous assertion.
Hence we may assume $a\in N_j$ for some $j\in \{4,...,q+1\}$.
Since $b\in V_1\backslash N(B)$ and $a\in N(b)\cap S_q$, by Claim \ref{Cla:R+L}, $\dR\cup \{N[a]\}$ is 1-intersecting.
Because $v_\ell\in A$ for $\ell\in \{3,4,...,q+1\}$, we have $|N[a]\cap N(v_\ell)|=1$.
This, together with Claim \ref{Cal:RcapV2} and the fact $N[a]\cap L_\2^v=\emptyset$, imply that $|N[a]\cap L_\1^v|\geq 2$.
As $av_1\notin E(G)$, we have $|N(a)\cap N(v_1)|\geq 2$, a contradiction as it would force a $C_4$.

Lastly, we observe that for every choice of such $b$, the above vertex $a$, which lies in $N_1\cap S_q\subseteq N(v_1)\cap S_q$, must be distinct.
This is because if there exist two vertices say $b_1,b_2$ corresponding to the same vertex $a$,
then it provides a $C_4$ such as $b_1ab_2cb_1$ in $G$.
There are at least $(1-\epsilon-3\delta)q-2/\delta-4$ choices for $b$, implying that
$|N(v_1)\cap S_q|\geq (1-\epsilon-3\delta)q-2/\delta-4\geq \delta q$,
where the last inequality holds as $1/q\ll\delta\ll 1-\epsilon$.
This shows $v_1\in B$, a contradiction to that $v\in S_{q+1}\backslash N(B)$, completing the proof.
\end{proof}

\begin{claim}\label{Cla:little type 2}
For any $u\in S_q\backslash B$, the number of non-extendable vertices $v\in V_2\cap N(u)$ with $u$ as the type-\2 vertex is at most $4\delta q$.
\end{claim}
\begin{proof}
Suppose on the contrary that there exists $u\in S_q\backslash B$ such that the set
$T=\{\mbox{non-extendable } v\in V_2\cap N(u): u \mbox{ is the type-\2 vertex of } v\}$ has size $t>4\delta q$.
We write $T=\{v_i: i\in [t]\}$ and for each $v_i\in T$, we denote the type-\1 vertex of $v_i$ by $u_i$.
It is clear that all $u_i$ are distinct (as otherwise it would force a $C_4$).

By Claim \ref{Cla:non-extendable-V2}, for $v_1$ there exists $L=N(a)\in \dR$ such that $|L\cap L_\1^{v_1}|=2$.
By Claim \ref{Cal:RcapV2}, $L\cap N(u)=\emptyset$.
This shows that exactly one of the following facts holds for $L$ and any other $v_i$'s in $T$:
\begin{itemize}
\item[(1).] If $|L\cap L_\1^{v_i}|=|L\cap L_\2^{v_i}|=1$, then $v_i^\star\in L$ (because $L_\2^{v_i}=N(u)\cup \{v_i^\star\}$ and $L\cap N(u)=\emptyset$);
\item[(2).] Otherwise $|L\cap L_\1^{v_i}|=2$, then $u_i\in L$ (because $L_\1^{v_i}=N[u_i]$ and $|L\cap N(u_i)|\leq 1$).
\end{itemize}
Let $T_1$ consist of all $v_i\in T$ with $|L\cap L_\1^{v_i}|=2$.
By (2), we see $u_i\in L=N(a)$ for all $v_i\in T_1$.
Since $a\in A$ and $u_i\in S$ are distinct, we have $|T_1|\leq |N(a)\cap S|<\delta q$.
Now let $L'=N(a')\in \dR$ be another line, other than $L$, such that $|L'\cap L_\1^{v_j}|=2$ for some $v_j\in T\backslash T_1$.
Let $T_2$ consist of all $v_j\in T\backslash T_1$ with $|L'\cap L_\1^{v_j}|=2$.
Similarly, we also have $|T_2|<\delta q$.

We also assert that there are at most $\delta q$ vertices $v_i\in T$ sharing a common $v^\star_i$ (denoted by $v^\star$).
Suppose $v_{j_1},...,v_{j_s}\in T$ share a common $v^\star$ and subject to this, $s$ is maximum.
We first find a line $L_0$ with $|L_0\cap L_\1^{v_{j_1}}|=2$, which also satisfies $L_0\cap (N(u)\cup \{v^\star\})=\emptyset$.
If $L_0$ satisfies that $|L_0\cap L_\1^{v_{j_i}}|=2$ for all $i\in [s]$,
then we can get that $s<\delta q$ as above.
So we may assume that there is some $v_{j_i}$ such that (1) holds.
By (1), we then derive that $v^\star\in L_0$, a contradiction to that $L_0\cap (N(u)\cup \{v^\star\})=\emptyset$.
This proves the assertion.

By our choice, (1) holds for each of $L=N(a)$ and $L'=N(a')$ and for any $v_i\in T\backslash (T_1\cup T_2)$.
Since $|T\backslash (T_1\cup T_2)|>2\delta q$ and at most $\delta q$ vertices $v_i\in T\backslash (T_1\cup T_2)$ share a common $v^\star_i$,
we can find two different $v^\star_i$, say $x$ and $y$.
By (1), we have $x,y\in N(a)\cap N(a')$.
This forces a $C_4$ in $G$ and finishes the proof.
\end{proof}

The next claim shows that the type of a good vertex $u\in S_q\backslash B$ in fact is an invariance
(that is, the type remains the same for all extendable vertices in $N(u)\cap V_2$).

\begin{claim} \label{Cla:good type}
Let $u\in S_q\backslash B$.
If $u$ is a good vertex of type-\1 for some vertex in $V_2$, then $u$ is the good vertex of type-\1 for all vertices in $N(u)\cap V_2$.

If $u$ is a good vertex of type-\2 for some vertex in $V_2$, then $u$ is a vertex of type-\1 (which must be bad) for at most one vertex in $N(u)\cap V_2$;
moreover, there are at most two good lines of type-\2 for all the vertices in $N(u)\cap V_2$.
\end{claim}

\begin{proof}
First let us assume that $u$ is a good vertex of type-\1 for some $v\in V_2$.
Since $v$ is extendable, $\dR\cup \{N[u]\}$ is 1-intersecting.
Consider any $v'\in N(u)\cap V_2\backslash \{v\}$.
Suppose that $u$ is a type-\2 vertex for $v'$.
If we let $x$ be the unique vertex in $N(v')$ adjacent to $u$,
then we see $N(x)\in \dR$ and $N(x)\cap N[u]=\{u,v'\}$, a contradiction to that $\dR\cup \{N[u]\}$ is 1-intersecting.
Thus $u$ is always type-\1 for all vertices in $N(u)\cap V_2$.
Let $N(v')=\{u,a_2,...,a_{q+1}\}$, where $a_2$ is type-\2 for $v'$ and $a_3,...,a_{q+1}\in A$.
Then Claim \ref{Cal:RcapV2} holds for $v'$ analogously, where $N(a_3),...,N(a_{q+1})$ and $L_\1^{v'}=N[u]$ are $q$ lines in the 1-intersecting hypergraph $\dR\cup \{N[u]\}$.
By Lemma \ref{lem:enlarge R}, we infer that $\dR\cup \{L_\1^{v'}, L_\2^{v'}\}$ is 1-intersecting and thus $v'$ is also extendable.
This shows that $u$ is a good vertex of type-\1 for all vertices in $N(u)\cap V_2$.

Now we assume that $u$ is a good vertex of type-\2 vertex for $v_1\in V_2$.
Let $F=N(u)\cup \{v_1^\star\}$. Since $v_1$ is extendable, $\dR\cup \{F\}$ is 1-intersecting.
Suppose that $v_2\in N(u)\cap V_2$ is distinct from $v_1$ and $u$ is a type-\1 vertex for $v_2$.
By the first assertion, we see $v_2$ is non-extendable and $u$ is bad for $v_2$.
So $N(v_2)\in \dR$ and $N(v_2)\cap N(u)=\emptyset$.
Since $\dR\cup \{F\}$ is 1-intersecting, the only possibility is that $v^\star_1\in N(v_2)$.
Thus we have $u,v^\star_1\in N(v_2)$.
Now suppose there exists another $v_3\in N(u)\cap V_2\backslash \{v_1,v_2\}$ which has $u$ as its type-\1 vertex.
Similarly we have $u,v^\star_1\in N(v_3)$, giving $uv_2v^\star_1v_3u$ as a $C_4$ in $G$, a contradiction.
Hence by our discussion, $u$ is a vertex of type-\1 for at most one vertex of $N(u)\cap V_2$.

By Claim \ref{Cla:N(V1^*)}, $u\in S_q\backslash (B\cup N(V_1^*))$.
Then the previous paragraph, together with Claims \ref{Cla:many prop 2} and \ref{Cla:little type 2},
show that $u$ appears as a good type-\2 vertex for at least $(1-\epsilon-3\delta)q-(1+6/\delta)-1-4\delta q\geq 20$ vertices $w_i\in N(u)\cap V_2$.
Suppose that there exist at least three good lines of type-\2 for these $w_i$'s, say $F_1=N(u)\cup \{w_1^\star\}$, $F_2=N(u)\cup \{w_2^\star\}$ and $F_3=N(u)\cup \{w_3^\star\}$.
By renaming notations if necessary, we may assume that there are at least four $w_i\in N(u)\cap V_2$ whose $w_i^\star\notin \{w_1^\star,w_2^\star\}$.
We observe that the type-\1 vertices $u_i$ of these $w_i$ are all distinct;
indeed, otherwise say $w_i$ and $w_j$ have the same type-\1 vertex $u'$, then it yields a $4$-cycle $w_iu'w_juw_i$, a contradiction.
Hence, we can further find two of these $w_i$, say $w_4,w_5\in N(u)\cap V_2$, such that their type-\1 vertices $u_4,u_5$ are distinct and not in $\{w_1^\star,w_2^\star\}$.
Now we show $w_1^\star \in N(u_4)$ by considering the location of $w_1^\star$ in the local structure based on $w_4$.
As $w_4\in V_2$, any $v\in N(w_4)\cap S_{q+1}$ is in $A$, so $w_4$ is the unique vertex in $N(v)\cap F_1$ and thus $w_1^\star\notin N(v)$.
Since $V=\left(\cup_{v\in N(w_4)\cap S_{q+1}} N(v)\right)\cup L_\1^{w_4}\cup L_\2^{w_4}$ and $w_1^\star\notin L_\2^{w_4}=N(u)\cup \{w_4^\star\}$,
we conclude that $w_1^\star \in L_\1^{w_4}=N[u_4]$, which further shows $w_1^\star \in N(u_4)$, as wanted.
Analogously, we can derive that $w_1^\star, w_2^\star\in N(u_4)\cap N(u_5)$.
This provides a $C_4$ in $G$, a contradiction.
\end{proof}

Claim \ref{Cla:good type} also indicates that the type of a good vertex $u\in S_q\backslash B$ is consistent with the type of all good lines containing $N(u)$.
Consequently, just the same as a good vertex, a good line can only be of one particular type.
To further study good lines, we define an auxiliary graph as following.

\begin{dfn}\label{dfn:dG}
We denote $\dF$ by the set of all good lines (of type-\1 or type-\2 for any vertex in $V_2$).
Let $\dG$ be the graph with vertex set $\dF$, where $F,F'\in \dF$ are adjacent if and only if they are the type-\1 and type-\2 lines for some extendable vertex in $V_2$.
\end{dfn}

By the above discussion, we see that $\dG$ is a bipartite graph on two parts $(\dF_\1,\dF_\2)$, where $\dF_\1$ consists of all good lines of type-\1 and $\dF_\2$ consists of all good lines of type-\2.

\begin{claim}\label{Cla:vector-of-4-pairs}
For two independent edges $(L_\1^v, L_\2^v)$ and $(L_\1^w, L_\2^w)$ in $\dG$,
we have $\overline{vw}=(|L_\1^v\cap L_\1^w|, |L_\1^v\cap L_\2^w|, |L_\2^v\cap L_\1^w|, |L_\2^v\cap L_\2^w|)\in \{(1,1,1,1), (2,0,0,2),(0,2,2,0),(0,2,2,q)\}$.
\end{claim}
\begin{proof}
Clearly $v,w\in V_2$ are extendable.
We define $v_1,...,v_{q+1}, v^\star$ and $w_1,...,w_{q+1}, w^\star$ according to Definition \ref{dfn:more-on-V2} for $v$ and $w$, respectively.

First we observe that if any entry in the vector $\overline{vw}$ is one, then by applications of Lemma \ref{lem:enlarge R}, we can infer that $\dR\cup \{L_\1^v, L_\2^v, L_\1^w, L_\2^w\}$ is $1$-intersecting and thus $\overline{vw}=(1,1,1,1)$.
Hence, we may assume that none of the entries in $\overline{vw}$ is one.

Suppose $v\in L_\1^w$. Then we have $v\in N(w_1)$ and thus $w_1\in \{v_1,v_2\}$.
If $w_1=v_1$, then $L_\1^w=L_\1^v$, a contradiction as the two edges in $\dG$ are independent.
So $w_1=v_2$. Then $\{v_2,v_3\}\subseteq N(v)\cap L_\1^w$, however $\{N(v),L_\1^w\}\subseteq \dR\cup \{L_\1^w\}$ is 1-intersecting, a contradiction.
Hence we must have $v\notin L_\1^w$.
Also we have $|L_\1^w\cap N(v_i)|=1$ for $3\leq i\leq q+1$.
So $|L_\1^w\cap (L_\1^v\cup L_\2^v)|=2$.
By symmetry between $v$ and $w$, we also can get $|L_\1^v\cap (L_\1^w\cup L_\2^w)|=2$.

If $|L_\1^v\cap L_\1^w|=2$, then $|L_\1^v\cap L_\2^w|=|L_\2^v\cap L_\1^w|=0$.
Note that $L_\1^v, L_\2^v, N(v_3),...,N(v_{q+1})$ form a sun-flower with center $v$ and vertex set $V$.
Since $L_\2^w\cap L_\1^v=\emptyset$ and $|L_\2^w\cap N(v_i)|=1$ for $3\leq i\leq q+1$,
we derive that $|L_\2^w\cap L_\2^v|=2$, i.e., $\overline{vw}=(2,0,0,2)$.

It remains to consider $|L_\1^v\cap L_\1^w|=0$. Then $|L_\1^v\cap L_\2^w|=|L_\2^v\cap L_\1^w|=2$.
If $v\notin L_\2^w$, using the same argument as above, we can easily get $\overline{vw}=(0,2,2,0)$.
Hence we assume $v\in L_\2^w$. If $v=w^\star$, then clearly $N(v)\cap N(w)=\emptyset$, contradicting that $\dR$ is 1-intersecting.
So $v\in N(w_2)$, implying that $w_2\in \{v_1,v_2\}$.
If $w_2=v_1$, as $N(w_2)\subseteq L_\2^w$ and $N(v_1)\subseteq L_\1^v$, we have a contradiction that $2=|L_\2^w\cap L_\1^v|\geq q$.
So $w_2=v_2$, which shows either $\overline{vw}=(0,2,2,q)$ or $L_\2^w=L_\2^v$ (a contradiction).
\end{proof}

We say two good lines $F, F'\in \dF$ are {\it friendly} if $|F\cap F'|=1$ and two components $\mathcal{D}$ and $\mathcal{D}'$ of $\dG$ are {\it friendly}
if there is a friendly pair $\{F,F'\}$ with $F\in V(\mathcal{D})$ and $F'\in V(\mathcal{D}')$.

\begin{claim}\label{Cla:how-be-friendly}
If two components $\mathcal{D}$ and $\mathcal{D}'$ of $\dG$ are {\it friendly} (possibly $\mathcal{D}=\mathcal{D}'$),
then all pairs in $\mathcal{D}\cup\mathcal{D}'$ are friendly.
In particular, any two vertices in the same component of $\dG$ are friendly.
\end{claim}

\begin{proof}
First we point out that if two edges say $(L_1,L_2)$ and $(L_1,L_3)$ in $\dG$ share a common vertex,
then using Lemma \ref{lem:enlarge R}, one can derive that $\dR\cup \{L_1,L_2,L_3\}$ is 1-intersecting.

Assume that there exist $F\in V(\mathcal{D})$ and $F'\in V(\mathcal{D}')$ with $|F\cap F'|=1$.
To complete the proof, it suffices to show that for any vertex $L\in V(\mathcal{D})$, we have $|L\cap F'|=1$ (unless $\mathcal{D}=\mathcal{D}'$ and $L=F'$).
We will prove this by induction on the length $d_L$ of the shortest path between $L$ and $F$ in $\mathcal{D}$.
If $d_L=0$, then $L=F$ and clearly it is true.
Now suppose that for any $L^*\in V(\mathcal{D})$ with $d_{L^*}\leq k$, the above statement holds.
Consider any $L\in V(\mathcal{D})$ with $d_{L}=k+1$.
Then there exists an edge $(L,L^*)$ in $\mathcal{D}$ with $d_{L^*}=k$.
By induction, we have $|L^*\cap F'|=1$, unless $\mathcal{D}=\mathcal{D}'$ and $L^*=F'$.
In the latter case, obviously we have $|L\cap F'|=1$.
Therefore $|L^*\cap F'|=1$.
Fix an edge incident to $F'$, say $(L',F')$ in $\mathcal{D}'$.
If $(L,L^*)$ and $(L',F')$ share a common vertex, then $\mathcal{D}=\mathcal{D}'$ and by the first paragraph, it is easy to see that either $L=F'$ or $|L\cap F'|=1$.
So we may assume $(L,L^*)$ and $(L',F')$ are two independent edges.
By Claim \ref{Cla:vector-of-4-pairs}, as $|L^*\cap F'|=1$, it infers that $|L\cap F'|=1$.
This finishes the proof.
\end{proof}

\begin{dfn}\label{dfn:rich-component}
For each $F\in \dF$, let $u_F$ be the unique vertex in $S_q\backslash B$ satisfying $N(u_F)\subseteq F$.
We say a component of $\dG$ is {\bf rich} if it contains some vertex $F$ with $u_F\in N(V_1^*)$.
\end{dfn}

We remark that for each $F\in \dF$, there exists some extendable $v_F\in V_2$ such that $u_F$ is a good vertex and $F$ is a good line for $v_F$ of the same type.

\begin{claim}\label{Cla:rich-component}
A rich component of $\dG$ is friendly with any component of $\dG$.
\end{claim}

\begin{proof}
Let $\mathcal{D}$ be a rich component of $\dG$ and let $F\in \dF$ be a vertex in $\mathcal{D}$ with $u_F\in N(V_1^*)$.
Then $u_F$ is a good vertex and $F$ is a good line for some $v\in V_2$ of the same type.
By Claim \ref{Cla:N(V1^*)}, $u_F$ is type-\1 for $v$ and $F=N[u_F]\in \dL$.
Now consider any component $\mathcal{D}'$ of $\dG$ and take any vertex $L\in V(\mathcal{D}')\cap\dF$.
Then $\dR\cup \{L\}$ is 1-intersecting.
Recall that Claim \ref{Cla:L-property} holds for $F\in \dL$.
By Lemma \ref{lem:enlarge R}, we can derive that $\dR\cup \{L,F\}$ is 1-intersecting.
This shows that $\mathcal{D}$ and $\mathcal{D}'$ are friendly.
\end{proof}

\begin{dfn}\label{dfn:large-component}
If a component of $\dG$, which is not rich, contains at least $40$ vertices in $\dF_\1$ and at least $\delta q$ vertices in $\dF_\2$, then we say it is {\bf heavy}.
\end{dfn}

\begin{claim}\label{Cla:size-large-component}
Let $\dF_r$ and $\dF_h$ denote two unions of vertices in rich components and heavy components of $\dG$, respectively.
Then $|\dF_r\cup \dF_h\cup \dL|\geq 2\alpha q +\frac12(1-\epsilon)q.$
\end{claim}

\begin{proof}
We note that $\dL=\{N[u]: u\in S_q\cap N(V_1^*)\}$ and $|\dL|=|S_q\cap N(V_1^*)|$.

For $u\in S_q\backslash (B\cup N(V_1^*))$, if all vertices in $N(u)\cap V_2$ are non-extendable, then we say $u$ is {\it poor}.
Let $P$ be the set of all poor vertices in $S_q\backslash (B\cup N(V_1^*))$.
Consider the number $M$ of pairs $(v,u)$, where $v\in V_2$ is non-extendable and $u\in S_q\backslash (B\cup N(V_1^*))$ is the type-\1 vertex of $v$.
Clearly $M$ is at most the number of pairs $(v,u')$, where $v\in V_2$ is non-extendable and $u'\in S_q\backslash B$ is the type-\2 vertex of $v$.
By Claims \ref{Cla:many prop 2} and \ref{Cla:little type 2},
$|P|\cdot\left((1-\epsilon-3\delta)q-(1+6/\delta)-4\delta q\right)\leq |M|\leq |S_q\backslash B|\cdot 4\delta q.$
This implies that
\begin{equation} \label{equ:poor}
|P|\leq \frac{4\delta q\cdot|S_q\backslash B|}{(1-\epsilon-3\delta)q-(1+6/\delta)-4\delta q}= \frac{4\delta q\cdot|S_q\backslash B|}{(1-\epsilon-7\delta) q-(1+6/\delta)}.
\end{equation}
We point out that any $u\in S_q\backslash (B\cup N(V_1^*)\cup P)$ has an extendable neighbor in $V_2$,
which implies at least one good line (not in $\dL$) containing $N(u)$.
By Claim \ref{Cla:good type},
if $u$ is type-\1, then the unique good line containing $N(u)$ is $N[u]$.
If $u$ is type-\2, then there are at most two good lines containing $N(u)$, say $F^u_1$ and $F^u_2$.
We call these good lines as the {\it associated} lines of $u\in S_q\backslash (B\cup N(V_1^*)\cup P)$.

Let $F$ be any line in $\dF_\1$ with $u_F\notin N(V_1^*)$.
Clearly $u_F\in S_q\backslash (B\cup N(V_1^*))$ is type-\1 and $F=N[u_F]$.
By Claims \ref{Cla:many prop 2} and \ref{Cla:good type},
$u_F$ is the good vertex of type-\1 for at least $(1-\epsilon-3\delta)q-(1+6/\delta)\geq \delta q$ vertices in $N(u_F)\cap V_2$.
This shows that there are at least $\delta q$ good lines of type-\2 adjacent to $F$ in $\dG$, i.e., $F$ has degree at least $\delta q$ in $\dG$.

Consider a type-\2 vertex $u\in S_q\backslash (B\cup N(V_1^*)\cup P)$.
As mentioned above, there are at least one and at most two associated lines say $F^u_1$ and $F^u_2$ of $u$.
By Claims \ref{Cla:many prop 2}, \ref{Cla:little type 2} and \ref{Cla:good type}, we see that $u$ is the good vertex of type-\2 for at least
$(1-\epsilon-3\delta)q-(1+6/\delta)-4\delta q-1\geq 80$ vertices in $N(u)\cap V_2$.
Thus at least one of $F^u_1$ and $F^u_2$ has at least $40$ neighbors in $\dF_\1$.
By the previous paragraph, we see that at least one associated line of $u$ is contained in a rich or heavy component of $\dG$.

Now we show that the number of type-\1 vertices $u\in S_q\backslash (B\cup N(V_1^*)\cup P)$,
which has no associated lines in any rich or heavy component of $\dG$, is at most $\frac{40 |S_q\backslash B|}{(1-\epsilon-3\delta)q-(1+6/\delta)}$.
If a component $\mathcal{D}$ of $\dG$ is neither rich nor heavy,
then each $F\in V(\mathcal{D})$ has $u_F\notin N(V_1^*)$ and it has no more than $40$ vertices in $\dF_\1$,
implying that each $F\in V(\mathcal{D})\cap \dF_\2$ has degree at most $40$.
Let $e^*, \ell_1$ and $\ell_2$ denote the numbers of edges, vertices of $\dF_\1$ and vertices of $\dF_\2$ contained in all components of $\dG$ which are neither rich nor heavy, respectively.
Then we have $\ell_1\cdot ((1-\epsilon-3\delta)q-(1+6/\delta))\leq e^*\leq 40\cdot \ell_2\leq 40 |S_q\backslash B|$.
This implies what we want.

Therefore we have $|(\dF_r\cup \dF_h)\backslash \dL|\geq |S_q\backslash (B\cup N(V_1^*)\cup P)|-\frac{40 |S_q\backslash B|}{(1-\epsilon-3\delta)q-(1+6/\delta)}.$
This, together with \eqref{equ:poor} and $|\dL|=|S_q\cap N(V_1^*)|$, imply that
\begin{equation*}
\begin{split}
|\dF_r\cup \dF_h\cup \dL|&\geq |S_q\backslash (B\cup N(V_1^*)\cup P)|-\frac{40 |S_q\backslash B|}{(1-\epsilon-3\delta)q-(1+6/\delta)}+|S_q\cap N(V_1^*)|\\
&\geq |S_q\backslash B|- \frac{4\delta q\cdot|S_q\backslash B|}{(1-\epsilon-7\delta) q-(1+6/\delta)}-\frac{40 |S_q\backslash B|}{(1-\epsilon-3\delta)q-(1+6/\delta)}.
\end{split}
\end{equation*}
By \eqref{equ:choices-of-q&delta} and \eqref{equ sq}, it shows that $|\dF_r\cup \dF_h\cup \dL|\geq 2\alpha q +\frac12(1-\epsilon)q$, completing the proof.
\end{proof}

\begin{claim}\label{Cla:half-of-heavy}
At least half of the lines in $\dF_h$ are pair-wise friendly.
\end{claim}
\begin{proof}
We define $H$ to be a graph whose vertices are heavy components of $\dG$, where components $\mathcal{C}, \mathcal{D}\in V(H)$ are adjacent if and only if they are not friendly.
To prove this, we may assume that any component has a non-friendly component in $H$ (as otherwise, we can delete it and process).
By Claim \ref{Cla:how-be-friendly}, it suffices to show that $H$ is a bipartite graph.

Let $\mathcal{C}$ and $\mathcal{D}$ be two heavy components of $\dG$ which are not friendly.
Take a type-\1 line $F\in V(\mathcal{D})$.
By Claim \ref{Cla:vector-of-4-pairs}, for any edge $e=(L_1,L_2)$ in $\mathcal{C}$,
$c(e):=(|F\cap L_1|,|F\cap L_2|)$ is either $(0,2)$ or $(2,0)$.
Repeatedly applying this, we see that all edges $e$ in $\mathcal{C}$ have the same $c(e)$.
So for any line $F\in V(\mathcal{D})\cap \dF_\1$, exactly one of the following holds:
\begin{itemize}
\item[(A).] All type-\1 lines $L$ in $\mathcal{C}$ satisfy $|F\cap L|=2$. Since $|N(u_F)\cap N(u_L)|\leq 1$, we have $u_Lu_F\in E(G)$.
\item[(B).] All type-\2 lines $L$ in $\mathcal{C}$ satisfy $|F\cap L|=2$. Suppose $L$ is a type-\2 line for $v_L\in V_2$. Then $L=N(u_L)\cup \{v_L^\star\}$ and we have either $v_L^\star=u_F$, $v_L^\star u_F\in E(G)$, or $u_Lu_F\in E(G)$.
\end{itemize}

\noindent Suppose (A) holds for some $F\in V(\mathcal{D})\cap \dF_\1$.
Repeatedly applying the above conclusion, one would derive that in fact any $L\in V(\mathcal{C})\cap \dF_\1$ and any $R\in V(\mathcal{D})\cap \dF_\1$ satisfy $|L\cap R|=2$ and $u_Lu_R\in E(G)$.
Since all $u_L, u_R$ are distinct, it is easy to force a $C_4$ in $G$, a contradiction.

Hence (B) holds for all type-\1 lines $F\in V(\dD)$ and any heavy component $\dC$ which is not friendly with $\dD$.
For every such $F$, we can partition $V(\dC)\cap \dF_\2$ into three sets $X(F),Y(F)$ and $Z(F)$,
where $X(F)=\{L\in V(\dC)\cap \dF_\2: v_L^\star=u_F\}$, $Y(F)=\{L\in V(\dC)\cap \dF_\2: v_L^\star u_F\in E(G)\}$ and $Z(F)=\{L\in V(\dC)\cap \dF_\2: u_Lu_F\in E(G)\}$.
Let $F,F'\in V(\mathcal{D})\cap \dF_\1$ be distinct.
Then it is easy to see that $|X(F)\cap X(F')|=0$ and $|Z(F)\cap Z(F')|\leq 1$.
If there are $L_1,L_2\in Y(F)\cap Y(F')$, then both $v_{L_1}^\star$ and $v_{L_2}^\star$ are adjacent to $\{u_F,u_{F'}\}$.
This shows that $v_{L_1}^\star=v_{L_2}^\star$, as otherwise there is a $C_4$ in $G$.
That is, all lines $L\in Y(F)\cap Y(F')$ have the same $v_L^\star$.
By $f_\dC(w)$, we denote the number of all lines $L\in V(\dC)\cap \dF_\2$ with $v_L^\star=w$.

We first assert that there exists some vertex $w$ with $f_\dC(w)\geq (|V(\dC)\cap \dF_\2|-6)/6$.
To see this, let us take four lines $F_1,F_2,F_3,F_4\in V(\mathcal{D})\cap \dF_\1$.
We have $|V(\dC)\cap \dF_\2|\geq |\cup_{1\leq i\leq 4} X(F_i)|=\sum_{1\leq i\leq 4} |X(F_i)|,$
and by inclusion-exclusion,
$$|V(\dC)\cap \dF_\2|\geq |\cup_{1\leq i\leq 4} Y(F_i)|\geq \sum_{1\leq i\leq 4}|Y(F_i)|-\sum_{1\leq i< j\leq 4}|Y(F_i)\cap Y(F_j)|$$
and similarly, $|V(\dC)\cap \dF_\2|\geq \sum_{1\leq i\leq 4}|Z(F_i)|-6$.
Summing up the above three inequalities, using $|X(F_i)|+|Y(F_i)|+|Z(F_i)|=|V(\dC)\cap \dF_\2|$ for each $i\in [4]$, we obtain that
$$3|V(\dC)\cap \dF_\2|\geq 4|V(\dC)\cap \dF_\2|-\sum_{1\leq i< j\leq 4}|Y(F_i)\cap Y(F_j)|-6.$$
Therefore one of $Y(F_i)\cap Y(F_j)$ contains at least $(|V(\dC)\cap \dF_\2|-6)/6$ lines $L$, all of which have the same $v_L^\star$.
This proves our assertion.

We further assert that in fact for each heavy component $\dC$, there exists a unique vertex $w_\dC$ with
$f_\dC(w_\dC)\geq 0.6\cdot |V(\dC)\cap \dF_\2|$.
Let $c=|V(\dC)\cap \dF_\2|$, which is at least $\delta q$.
We choose $w$ such that $f_\dC(w)$ is maximum. So we have $f_\dC(w)\geq (c-6)/6$.
Take any $20$ lines say $F_1,...,F_{20}$ in $V(\dD)\cap \dF_\1$.
Then we have $20c=\sum_{1\leq i\leq 20} (|X(F_i)|+|Y(F_i)|+|Z(F_i)|)\leq \sum_{1\leq i\leq 20} |Y(F_i)| +2c+\binom{20}{2}$,
implying that $\sum_{1\leq i\leq 20} |Y(F_i)|\geq 18c-\binom{20}{2}.$
Without loss of generality we may assume $wu_{F_i}\in E(G)$ if and only if $i\in [s]$ for some $s\in [20]$.
Note that for $s<i\leq 20$, we have $|Y(F_i)|\leq c-f_c(w)$.
If $s\leq 10$, then $\sum_{1\leq i\leq 20} |Y(F_i)|\leq 20c-(20-s)\cdot f_\dC(w)\leq 20c-10(c-6)/6<18c-\binom{20}{2},$ a contradiction.
Hence $s\geq 11$. For $i\in [s]$, let $A_i=N_G(u_{F_i})\backslash \{w\}$ and let $f_i=\sum_{x\in A_i} f_\dC(x)$.
We claim that $A_i$'s are disjoint over $i\in [s]$;
as otherwise say $x\in A_1\cap A_2$, then both $u_{F_1},u_{F_2}$ are adjacent to $w,x$ in $G$, a contradiction.
So $\sum_{i\in [s]} f_i\leq c-f_\dC(w)$.
Then we can derive that $18c-\binom{20}{2}\leq \sum_{1\leq i\leq 20} |Y(F_i)|\leq s\cdot f_\dC(w)+\sum_{i\in [s]} f_i+(20-s)c\leq (s-1)f_\dC(w)+(21-s)c\leq 10f_\dC(w)+10c,$
implying that $f_\dC(w)\geq (8c-\binom{20}{2})/10\geq 0.6c$, as desired.
Clearly such $w$ is unique, denoted by $w_\dC$.

We say $w_\dC$ is the {\it associated} vertex of $\dC$.
We also point out that from the above proof, among any $20$ lines $F_1,...,F_{20}$ in $V(\dD)\cap \dF_\1$,
the vertex $w_\dC$ is adjacent to at least $11$ of $u_{F_i}$'s in $G$.
Consider any two incident edges in $H$, say $\dC\dC_1, \dC\dC_2\in E(H)$.
Let $w_1,w_2$ be the associated vertices of $\dC_1, \dC_2$, respectively.
Take any $20$ lines $F_1,...,F_{20}$ in $V(\dC)\cap \dF_\1$.
Then each of $w_1$ and $w_2$ is adjacent to at least $11$ of $u_{F_i}$'s in $G$.
So there are two vertices $u_{F_i}, u_{F_j}$ adjacent to both $w_1$ and $w_2$ in $G$.
If $w_1\neq w_2$, then this forces a $C_4$, a contradiction.
So we conclude that $w_1=w_2$.

Now suppose on the contrary that $H$ contains an odd cycle, say $\dC_1\dC_2...\dC_t\dC_1$, where $t$ is odd.
Let $w_i$ be the associated vertex of $\dC_i$ for each $i\in [t]$.
Applying the above conclusion, as $t$ is odd, we can derive that these $w_i$'s are the same vertex, say $w$.
Take any $L_1\in V(\dC_1)\cap \dF_\2$ and $L_2\in V(\dC_2)\cap \dF_\2$ with $v_{L_1}^\star=v_{L_2}^\star=w$.
Then $L_1=N(u_{L_1})\cup \{w\}$ and $L_2=N(u_{L_2})\cup \{w\}$, implying $1\leq |L_1\cap L_2|\leq 2$.
If $|L_1\cap L_2|=1$, then $\dC_1$ and $\dC_2$ are friendly, a contradiction.
Hence $|L_1\cap L_2|=2$.
Let $L_1R_1$ be an edge in $\dC_1$ and $L_2R_2$ be an edge in $\dC_2$, where $R_1,R_2$ are type-\1.
By Claim \ref{Cla:vector-of-4-pairs}, as $|L_1\cap L_2|=2$, we have $|R_1\cap R_2|=2$.
So $(A)$ holds for the line $R_1\in V(\dC_1)\cap \dF_\1$ and the heavy component $\dC_2$, a contradiction.
This proves that $H$ is bipartite, completing the proof of Claim \ref{Cla:half-of-heavy}.
\end{proof}

Finally, we show how to add certain good lines into $\dR$ to make a larger 1-intersecting hypergraph on the vertex set $V$.
By Claim \ref{Cla:half-of-heavy}, there exists a 1-intersecting $\dF_h'\subseteq \dF_h$ with $|\dF_h'|\geq |\dF_h|/2$.
By Claim \ref{Cla:rich-component}, we see that $\dR\cup \dF_r\cup \dF_h'$ is also 1-intersecting.
Since $L\in \dL$ satisfies Claim \ref{Cla:L-property}, by Lemma \ref{lem:enlarge R}, we see that $\dR\cup \dF_r\cup \dF_h'\cup \dL$ remains 1-intersecting.
Using Claim \ref{Cla:size-large-component}, we have
\begin{equation*}
\begin{split}
|\dR\cup \dF_r\cup \dF_h'\cup \dL|&=|\dR|+|\dF_r\cup \dF_h'\cup \dL|\geq |\dR|+\frac12|\dF_r\cup \dF_h\cup \dL|\\
&\geq (q^2-\alpha q+1-2/\delta)+\frac12(2\alpha q +(1-\epsilon)q/2)\geq q^2.
\end{split}
\end{equation*}
This finishes the proof of Lemma \ref{Lem:finding-1-intersecting}. \QED

\section{Finding a polarity}\label{sec:polarity}
In this section, we complete the proof of Theorem \ref{Thm:stability-max-degree}.
Fix $0<\epsilon<1$ and let $q$ be an even integer with $q\gg \epsilon$.
Let $G=(V,\dE)$ be a $C_4$-free graph on $q^2+q+1$ vertices with maximum degree $q+1$ and at least $\frac12 q(q+1)^2-\frac{\epsilon}{2}q$ edges.
We aim to show that there exists a unique polarity graph of order $q$ containing $G$ as a subgraph.

By Lemma \ref{Lem:finding-1-intersecting}, there exists a 1-intersecting $(q+1)$-hypergraph $\dR^*$ on the vertex set $V$ with $\dR\subseteq \dR^*$ and $|\dR^*|\geq q^2$.
By Theorem \ref{Thm:Embedding}, $\dR^*$ (and thus $\dR$) can be embedded into a projective plane $\dP$ of order $q$.
Since $|\dR|=|A|\geq q^2-\epsilon q-2/\delta> q^2-q+1$ (by Claim \ref{Cla:|A|-|B|}),
applying Theorem \ref{Thm:Embedding unique}, we see that such $\dP$ and the embedding of $\dR$ into $\dP$ both are unique.
Let $\dR^c=\dP\backslash \dR$.
Now let us recall some basic facts about $\dP$:
every two lines intersect with exactly one vertex, every two vertices are contained in exactly one line,
and every column or row of any incidence matrix of $\dP$ has $q+1$ 1's as entries.

We say $v\in V$ is {\it feasible}, if there exists a line $L\in \dP$ with $N(v)\subseteq L$; otherwise, we say $v$  is {\it non-feasible}.
For non-feasible $v$, we say it is {\it near-feasible}, if there exist a line $L\in \dR^c$ and a subset $K_v\subseteq N(v)$ such that $N(v)\backslash K_v \subseteq L$ and $|K_v|\leq 5\sqrt{q}/\delta$.
In both definitions, we say $v$ and $L$ are {\it associated} with each other. For feasible $v$, we let $K_v=\emptyset$.
By \eqref{equ:dx+dy} and since $G$ is $C_4$-free, for any two feasible or near-feasible vertices $u$ and $v$, we have
$$|(N(u)\backslash K_u)\cup (N(v)\backslash K_v)|\geq \big(d(u)-5\sqrt{q}/\delta\big)+\big(d(v)-5\sqrt{q}/\delta\big)-1\geq (2-\epsilon)q-10\sqrt{q}/\delta-1>q+1.$$
This implies that each line in $\dP$ is associated with at most one feasible or near-feasible vertex.
On the other hand, if there are two lines in $\dP$ associated with the same feasible or near-feasible vertex $v$,
as $d(v)\geq (1-\epsilon)q$ by \eqref{equ:dx+dy}, then it is easy to see that these two lines will intersect with more than two vertices, a contradiction.
So each feasible or near-feasible vertex is associated with a unique line in $\dP$.

Next we study some properties on non-feasible vertices $v\in V$.
Let $N(v)=\{v_1,...,v_d\}$.
Since $v$ is non-feasible, we see $N(v)\not\subseteq L$ for any $L\in \dP$ and thus $v\notin A$.
Then any pair $\{v_i,v_j\}$ for $i, j\in [d]$ is not contained in any line $N(u)\in \dR$.
This is because that otherwise, we see that $v_iuv_jvv_i$ forms a $C_4$ in $G$, a contradiction.
So every such pair $\{v_i,v_j\}$ is contained in a unique line $L\in \dR^c$.
Let $\dL_v$ be the family of lines $L\in \dP$ which contains at least two vertices of $N(v)$.
Then we have $\dL_v\subseteq \dR^c$ and thus
\begin{equation}\label{equ:dLv}
|\dL_v|\leq |\dR^c|=|\dP|-|\dR|\leq (1+\epsilon)q+2/\delta+1.
\end{equation}
We also point out that any vertex in $N(v)$ appears in at least two lines of $\dL_v$.

We process to show that all non-feasible vertices are near-feasible in the following claims.
First we show any vertex has a neighbor which belongs to many lines in $\dR$.

\begin{claim}\label{Cla:dR(vj)}
Any vertex $v\in V$ has a neighbor $v_j$ with $d_\dR(v_j)=|N(v_j)\cap A|\geq q-\frac{\epsilon}{1-\epsilon}-2/\delta-2$.
In addition, if $v\notin B$ has degree at least $\frac12(1+\epsilon+4\delta)q+6/\delta+1$, then $v$ has a neighbor $v_j$ with $d_{\dR}(v_j)\geq q-1$.
\end{claim}
\begin{proof}
Let $N(v)=\{v_1,...,v_d\}$. By \eqref{equ:dx+dy}, we have $d=d(v)\geq (1-\epsilon)q$.
Let $N_i=N(v_i)\backslash N[v]$ for $i\in [d]$.
Since the sets $N_i\cup \{v_i\}$ are disjoint over $i\in [d]$, we have
\begin{align*}
q+\epsilon q+1\geq f(V)\geq\sum_{i\in [d]}f(N_i\cup \{v_i\})+f(v)=\sum_{i\in [d]}f(N_i\cup \{v_i\})+(q+1-d).
\end{align*}
By averaging, there is some $j\in [d]$ with $f(N_j\cup \{v_j\})\leq \frac{\epsilon q}{d}+1\le \frac{\epsilon}{1-\epsilon}+1$.
Therefore,
\begin{align*}
d_{\dR}(v_j)&\geq |N_j\cap A|\geq |N_j|-|N_j\cap S|-|B|\geq \big(d(v_j)-2\big)-f(N_j)-2/\delta\\
&=\big(q-1-f(v_j)\big)-f(N_j)-2/\delta\geq q-\epsilon/(1-\epsilon)-2/\delta-2,
\end{align*}
as desired.
Next we consider vertices $v\notin B$ with $d=d(v)\geq \frac12(1+\epsilon+4\delta)q+6/\delta+1$.
Let $B_v=N(v)\cap (S\cup B)$ and $B_v^*=B_v\cup (N(B_v)\cap N(v))$.
Then we have $|B_v|\leq |N(v)\cap S|+ |B|\leq \delta q+2/\delta$.
Since $G$ is $C_4$-free, every vertex in $B_v$ has at most one neighbor in $N(v)$, implying that $|B_v^*|\leq 2|B_v|$.
Let $T=\{v_i\in N(v)\backslash B_v^*: N_i\cap B=\emptyset\}$.
Since $N_i$'s are disjoint and there are at most $|B|$ many $N_i$'s containing some vertex in $B$,
we get $|T|\geq |N(v)\backslash B_v^*|-|B|\geq d-2\delta q -6/\delta$.
If $f(N_j)\geq 2$ for all $v_j\in T$, then $q+\epsilon q+1\geq f(V)\geq 2|T|\ge 2(d-2\delta q-6/\delta)\geq q+\epsilon q+2$, a contradiction.
Therefore, there exists a vertex $v_j\in T$ such that $f(N_j)\leq 1$.
By the definition of $T$, we can see that  $d_{\dR}(v_j)=|N(v_j)\cap A|\geq d(v_j)-1-f(N_j)\geq q-1$.
This completes the proof.
\end{proof}

We partition $V$ into three disjoint sets $U_1\cup U_2\cup U_3$, where $U_1$ consists of all feasible vertices and $U_2$ consists of non-feasible vertices $v\notin B$ with $d(v)\geq \frac12(1+\epsilon+4\delta)q+6/\delta+1$.

\begin{claim} \label{Cla:U2}
There exists some $w\in V$ such that all $v\in U_2$ are near-feasible with $K_v=\{w\}$.
\end{claim}
\begin{proof}
For any $v\in U_2$, by Claim \ref{Cla:dR(vj)}, there is a neighbor $v_j$ of $v$ with $d_\dR(v_j)\geq q-1$.
By the above discussion, $v_j$ appears in at least two lines in $\dL_v\subseteq \dR^c$.
If $d_{\dR}(v_j)\geq q$, then $d_{\dP}(v_j)\geq q+2$, a contradiction.
So $d_{\dR}(v_j)=q-1$ and there are exactly two lines, say $L_1$ and $L_2$, in $\dL_v\subseteq \dR^c$ containing $v_j$.
Let $N_1=L_1\cap N(v)$ and $N_2=L_2\cap N(v)$.
Then we have $N_1\cap N_2=\{v_j\}$ and $N_1\cup N_2=N(v)$.
Consider any other line $L_i\in \dL_v\backslash \{L_1,L_2\}$ for $i\geq 3$.
Set $N_i=L_i\cap N(v)$.
We see that for any $i\geq 3$ and $j\in \{1,2\}$, $|N_i\cap N_j|\leq 1$ and $|N_i\cap N_1|+|N_i\cap N_2|\geq |N_i\cap (N_1\cup N_2)|=|N_i|\geq 2$.
This shows that for any $i\geq 3$, $N_i$ consists of two vertices, one from $N_1\backslash \{v_j\}$ and the other from $N_2\backslash \{v_j\}$.
Hence, $|\dL_v|=(|N_1|-1)(|N_2|-1)+2$.

Let $d=d(v)$. We may assume that $d-1\geq |N_1|\geq |N_2|\geq 2$.
If $|N_2|\geq 3$, then we have $|\dL_v|=(|N_1|-1)(|N_2|-1)+2\geq 2(d-3)+2=2d-4\geq (1+\epsilon+4\delta)q+12/\delta-2>(1+\epsilon)q+2/\delta+1\geq |\dL_v|$,
where the last inequality holds by \eqref{equ:dLv}, a contradiction.
Thus, $|N_1|=d-1$ and $|N_2|=2$, implying $|\dL_v|=d$.
Suppose that $N_2=\{v_j,w\}$.
Then every $N_i$ for $2\leq i\leq d$ contains the vertex $w$.
Also $N(v)\backslash \{w\}\subseteq L_1\in \dR^c$,
implying that $v\in U_2$ is near-feasible with $K_v=\{w\}$.

Assume there is another non-feasible vertex $v'\in U_2$ with $K_{v'}=\{w'\}$, where $w'\neq w$.
Let $d=d(v)$ and $d'=d(v')$.
By the above arguments, $w$ and $w'$ appear in $d-1$ and $d'-1$ lines in $\dR^c$, respectively.
By \eqref{equ:dLv},
$|\dR^c|+2\leq (1+\epsilon)q+2/\delta+3\leq (1+\epsilon+4\delta)q+12/\delta\leq (d-1)+(d'-1)$,
which shows that $w$ and $w'$ appear in at least two lines of $\dR^c$ in common.
This contradicts that $\dP$ is a projective plane.
\end{proof}

\begin{claim} \label{Cla:U_3}
All non-feasible vertices are near-feasible.
\end{claim}
\begin{proof}
Let $v\in V$ be any non-feasible vertex. We have $d(v)\geq (1-\epsilon )q$.
By Claim \ref{Cla:dR(vj)}, $v$ has a neighbor $u$ with $d_\dR(u)=q+1-m$, where $m\leq \frac{\epsilon}{1-\epsilon}+2/\delta+3$.
Let $\mathcal{U}=\{L\in \dL_v: u\in L\cap N(v)\}$.
We have $|\mathcal{U}|\leq m$ and $\cup_{L\in \mathcal{U}} N_L=N(v)$, where $N_L:=L\cap N(v)$.
We assert that for all but at most one $L\in \mathcal{U}$, the size of $N_L$ is at most $2\sqrt{q}$.
Suppose on the contrary that there are $L_1, L_2\in \mathcal{U}$ with $|N_{L_1}|\geq 2\sqrt{q}+1$ and $|N_{L_2}|\geq 2\sqrt{q}+1$.
Then all pairs $(x,y)$ with $x\in N_{L_1}\backslash\{u\}$ and $y\in N_{L_2}\backslash\{u\}$ should appear in distinct lines in $\dL_v$.
By \eqref{equ:dLv}, this shows that $(1+\epsilon)q+2/\delta+1\geq |\dL_v|\geq (|N_{L_1}|-1)(|N_{L_2}|-1)\geq 4q$, a contradiction.

Let $L_1$ be the line in $\mathcal{U}$ with the maximum $N_{L_1}$ and let $K_v=\bigcup_{L\in \mathcal{U}\backslash \{L_1\}}(N_L\backslash \{u\})$.
Then $N(v)\backslash K_v\subseteq L_1\in \dR^c$ with
$|K_v|\leq \sum_{L\in \mathcal{U}\backslash \{L_1\}}(|N_L|-1)\leq (m-1)\cdot 2\sqrt{q}\leq 2\left(\frac{\epsilon}{1-\epsilon}+\frac2{\delta}+2\right)\sqrt{q}\leq \frac{5\sqrt{q}}{\delta}.$
Therefore, $v$ is near-feasible.
\end{proof}

We express $V=\{v_1,...,v_n\}$ such that $U_1=\{v_1,...,v_a\}$, $U_2=\{v_{a+1},...,v_b\}$ and $U_3=\{v_{b+1},....,v_n\}$ for $1\leq a<b\leq n$.
Since all vertices in $G$ are feasible or near-feasible, by the discussion before Claim \ref{Cla:dR(vj)},
we can conclude that each $v_i\in V$ is associated with a unique line denoted by $L_i$ in $\dP$.
Let $\pi: V\leftrightarrow \dP$ be a function which maps $v_i\leftrightarrow L_i$ for every $i\in [n]$.
Let $\dM=(m_{ij})$ be the incidence matrix of $\dP$ with respect to $\pi$.

Let $s:=|U_3|$. We point out that any $v\in U_3$ either is in $B$ or has $d(v)\leq (1+\epsilon+4\delta)q/2+6/\delta$.
In the latter case, we have the deficiency $f(v)=q+1-d(v)\geq (1-\epsilon-4\delta)q/2-6/\delta$.
So by \eqref{equ:new fV}, we see $$s\leq |B|+\frac{f(V)}{(1-\epsilon-4\delta)q/2-6/\delta}\leq \frac{2}{\delta}+\frac{q+\epsilon q+1}{(1-\epsilon-4\delta)q/2-6/\delta}\leq C^*,$$
where $C^*$ is a constant depending on $\epsilon$ and $\delta$ only.
Let $K$ be the union of $K_v$'s over all $v\in V$.
By Claims \ref{Cla:U2} and \ref{Cla:U_3}, we know that $K_v=\emptyset$ for $v\in U_1$, $K_v=\{w\}$ for $v\in U_2$ and $|K_v|\leq 5\sqrt{q}/\delta$ for $v\in U_3$.
Hence $|K|\leq 1+s\cdot 5\sqrt{q}/\delta=O(\sqrt{q})$.

\begin{claim}\label{final}
$\mathcal{M}$ is symmetric.
\end{claim}
\begin{proof}
We assert that if $v_i\in A\backslash K$, then $m_{ij}=m_{ji}$ for all $j\in[n]$.
If $m_{ij}=1$, then as $v_i\in A$, we have $v_j\in L_i=N(v_i)\in \dR$.
Since $v_i\notin K$, we see $v_i\in N(v_j)\backslash K\subseteq N(v_j)\backslash K_{v_j}\subseteq L_j$,
which shows that $m_{ji}=1=m_{ij}$.
Now we observe that as $v_i\in A$,
the $i$'{th} column and the $i$'{th} row of $\dM$ have exactly $q+1$ many 1-entries, and all these 1-entries are in the symmetric positions.
This shows that the $i$'{th} column and the $i$'{th} row are symmetric, proving the assertion.	
Since $|A\backslash K|\geq |A|-|K|\geq (q^2-\epsilon q-2/\delta)-O(\sqrt{q})\geq q^2-q+3$,
by Lemma \ref{Lem:Furedi-symmetry}, the whole matrix $\mathcal{M}$ is symmetric.
\end{proof}

Hence we see from \eqref{equ:polairty} that the above function $\pi: V\leftrightarrow \dP$ is a polarity of the projective plane $\dP$.
Let $H$ be the polarity graph of $\pi$.
For any $k\times\ell$ matrices $\mathcal{X}=(x_{ij})$ and $\mathcal{Y}=(y_{ij})$,
we say $\mathcal{X}$ {\it is at most} $\mathcal{Y}$ if $x_{ij}\leq y_{ij}$ for all $i,j$ and we express this by $\mathcal{X}\leq \mathcal{Y}$.

Now we are finishing the proof of Theorem \ref{Thm:stability-max-degree} by showing that $G$ is a subgraph of $H$.
Let $\mathcal{A}=(a_{ij})$ be the adjacent matrix of the graph $G$.
It suffices to shows that $\mathcal{A}\leq \dM$.
We call these $(i,j)$-entries with $a_{ij}=1$ and $m_{ij}=0$ {\it problematic}.
Since both $\mathcal{A}$ and $\dM$ are $0/1$ matrices, it is equivalent for us to show that there is no problematic entries.

For every $v_i\in U_1$, as it is feasible, we see that $N(v_i)\subseteq L_i$ and
thus the $i$'th row of $\mathcal{A}$ is at most the $i$'th row of $\dM$.
Since both $\mathcal{A}$ and $\dM$ are symmetric,
the $i$'th column of $\mathcal{A}$ is also at most the $i$'th column of $\dM$, whenever $v_i\in U_1$.
Now consider vertices $v_i\in U_2$.
By Claim \ref{Cla:U2}, $N(v_i)\backslash \{w\}\subseteq L_i$, where $w=v_{\ell}$ is fixed.
Consider $a_{ij}=1$ for possible $j$ which is not $\ell$.
Then we have $v_j\in N(v_i)\backslash \{w\}\subseteq L_i$.
This shows that the $i$'th row of $\mathcal{A}$ is at most the $i$'th row of $\dM$, except the $(i,\ell)$-entry.
By symmetry, we see that for all $v_i\in U_2$, the $i$'th column of $\mathcal{A}$ is at most the $i$'th column of $\dM$,
except the possible $(\ell,i)$-entry.
We also know $w$ is feasible or near-feasible.
So $|K_w|\leq 5\sqrt{q}/\delta$ and the number of problematic $(\ell,i)$-entries is clearly at most $|K_w|\leq 5\sqrt{q}/\delta$.
This further shows that the number of problematic $(i,j)$- or $(j,i)$-entries for all $v_i\in U_2$ is at most $10\sqrt{q}/\delta$.
Note that $|U_3|=s$ is at most a constant $C^*$ depending on only $\epsilon$ and $\delta$. Putting all the above together,
we see that the number of problematic $(i,j)$-entries for $i,j\in [n]$ is at most $10 \sqrt{q}/\delta+s^2=O(\sqrt{q})$.

Let $E_0$ be the set of $v_iv_j$ for all problematic $(i,j)$-entries.
It is easy to see that $E_0=E(G)\backslash E(H)$ and $|E_0|=O(\sqrt{q})$.
Suppose that there is some edge say $e=v_iv_j\in E_0$.
By Lemma \ref{Lem:polarity+plus}, $H\cup \{e\}$ contains at least $q-1$ copies of $C_4$,
all of which contain the edge $e$ and are edge-disjoint otherwise.
Hence in order to turn $H\cup \{e\}$ into a subgraph of $G$ containing $e$ (which is $C_4$-free),
one needs to delete at least $q-1$ edges in $H\cup \{e\}$.
On the other hand, since $H$ is a polarity graph, we have $e(H)\leq \frac12q(q+1)^2$ and $|E(H)\backslash E(G)|-|E_0|=e(H)-e(G)\leq \frac12\epsilon q$.
So one can delete $|E(H)\backslash E(G)|\leq \frac12\epsilon q+|E_0|\leq \frac12\epsilon q+ O(\sqrt{q})< q-1$ edges to turn $H\cup \{e\}$ into a subgraph of $G$ while preserving the edge $e$.
This is a contradiction.
Therefore, $E_0=\emptyset$ and $G$ is a subgraph of $H$.

It only remains to show that the polarity graph $H$ is unique.
Recall that the projective plane $\dP$ containing $\dR$ has been shown to be unique.
So it is equivalent to show that the polarity $\pi$ is unique.
Suppose for a contradiction that there exists another polarity $\pi': V\leftrightarrow \dP$,
where $\pi': v_i\leftrightarrow L_{\sigma(i)}$ for some permutation $\sigma$ on $[n]$.
Let $\dM'=(m'_{ij})$ be the incidence matrix of $\dP$ with respect to $\pi'$.
By the same proof as above, we can deduce that $\dA\leq \dM'$.
By \eqref{equ:dx+dy}, we see that any vertex $v_i\in V$ has degree at least $(1-\epsilon)q\geq 2$.
Choose any pair $\{x_i,y_i\}\subseteq N(v_i)$.
Since the $i$'th row of $\dA$ is at most the $i$'th row of $\dM'$,
we see $\{x_i,y_i\}\subseteq N(v_i)\subseteq L_{\sigma(i)}\in \dP$.
Also we have $\{x_i,y_i\}\subseteq N(v_i)\subseteq L_i\in \dP$.
Since $\dP$ is a projective plane, it is clear that $L_{\sigma(i)}=L_i$ for all $i\in [n]$.
This shows that $\pi=\pi'$ and indeed the polarity graph $H$ is unique.
The proof of Theorem \ref{Thm:stability-max-degree} (and thus Theorem \ref{Thm:stability}) is completed.
\QED

\medskip

We remark that it would suffice to choose $q_\epsilon =\frac{10^{10}}{(1-\epsilon)^6}$ in the statement of Theorem \ref{Thm:stability-max-degree}.

\section{Tur\'an numbers}\label{sec:Turan}
In this section, we discuss the consequences of Theorem \ref{Thm:stability} on Tur\'an numbers.
First, let us restate and prove Corollary \ref{coro:Turan1}. Recall the definition of $\lambda(q)$.

\begin{corollary}
Let $q$ be even. If $\lambda(q)\geq \frac12 q(q+1)^2-\frac{1}{2}q+o(q)$, then $\ex(q^2+q+1,C_4)=\lambda(q)$,
where the equality holds only for polarity graphs of order $q$ with $\lambda(q)$ edges;
otherwise, $\ex(q^2+q+1,C_4)< \frac12 q(q+1)^2-\frac{1}{2}q+o(q)$.
In particular, $\ex(q^2+q+1,C_4)\leq \max \left\{\lambda(q), \frac12 q(q+1)^2-\frac12 q+o(q)\right\}$.
\end{corollary}

\begin{proof}
Let $q$ be even and $G$ be an extremal graph for $\ex(q^2+q+1,C_4)$.
First suppose that $\lambda(q)\geq \frac12 q(q+1)^2-\frac{1}{2}q+o(q)$.
As $e(G)\geq \lambda(q)\geq \frac12 q(q+1)^2-\frac{1}{2}q+o(q)$, by Theorem \ref{Thm:stability},
there exists a polarity graph $H$ of order $q$ containing $G$ as a subgraph.
Then we have $\lambda(q)\leq e(G)\leq e(H)\leq \lambda(q)$,
which implies that $G=H$ must be a polarity graph of order $q$ with $\lambda(q)$ edges.
Now assume $\lambda(q)<\frac12 q(q+1)^2-\frac{1}{2}q+o(q)$.
By Theorem \ref{Thm:stability}, it is easy to conclude that $e(G)<\frac12 q(q+1)^2-\frac{1}{2}q+o(q)$.
\end{proof}

A quick inference of this corollary is that:
For all even integers $q$ such that there is no projective planes of order $q$, it holds that
$$\ex(q^2+q+1,C_4)\leq \frac12 q(q+1)^2-\frac12 q+o(q).$$
We point out that by Theorem \ref{Thm:BR}, there are infinitely many such integers $q$,
including all integers $q\equiv 2\mod 4$ which cannot be expressed as a sum of two square numbers.

Another inference can be stated related to the existence of orthogonal polarity graphs of order $q$.

\begin{corollary}\label{coro:Turan2}
Let $q$ be a large even integer. If there exists an orthogonal polarity graph of order $q$, then $\ex(q^2+q+1,C_4)=\frac12 q(q+1)^2$;
otherwise, we have $\ex(q^2+q+1,C_4)\leq \frac12 q(q+1)^2-\frac12\sqrt{q}$ and in addition if $q$ is not a square number, then $\ex(q^2+q+1,C_4)\leq \frac12 q(q+1)^2-\frac{1}{2}q+o(q)$.
\end{corollary}
\begin{proof}
By Proposition \ref{prop:G(pi)}, any polarity graph of order $q$ has $\frac12 q(q+1)^2-\frac12 m\sqrt{q}$ edges for some integer $m\geq 0$.
The first assertion follows by Theorem \ref{Thm:Furedi}.
Now we may assume that $m\geq 1$ for any polarity graph of order $q$ and thus $\lambda(q)\leq\frac12 q(q+1)^2-\frac12\sqrt{q}$.
By Corollary \ref{coro:Turan1}, $\ex(q^2+q+1,C_4)\leq \max \left\{\lambda(q), \frac12 q(q+1)^2-\frac12 q+o(q)\right\}\leq \frac12 q(q+1)^2-\frac12\sqrt{q}$.
In addition, if $q$ is not a square number, then this implies that there is no polarity graphs of order $q$ and the conclusion follows easily.
\end{proof}

We conclude this section with an explicit lower bound of $\ex(n,C_4)$ for later use.

\begin{prop}\label{prop:lower-Turan}
For any sufficiently large $n$, there exists some prime number $p$ with $\sqrt{n}-n^{0.2625}-1\leq p\leq \frac12(-1+\sqrt{4n-3})$ such that
$\ex(n,C_4)\geq \ex(p^2+p+1,C_4)\geq \frac{1}{2}(n^{1.5}-3n^{1.2625}+n)$.
\end{prop}

\begin{proof}
Let $x=\frac12(-1+\sqrt{4n-3})$.
As $n$ is sufficiently large, by Theorem \ref{Thm:prime exists}, there exists a prime number $p\in [x-x^{0.525},x]$.
So $n=x^2+x+1\geq p^2+p+1$, where $p\geq x-x^{0.525}\geq \sqrt{n}-n^{0.2625}-1$.
By Theorem \ref{Thm:Furedi}, we derive
$\ex(n,C_4)\geq \ex(p^2+p+1,C_4)= \frac{1}{2}p(p+1)^2\geq \frac{1}{2}(n^{1.5}-3n^{1.2625}+n).$
\end{proof}

\section{Supersaturation: one additional edge}\label{sec:supert=1}
Here, we extend Theorem \ref{Thm:supersaturation-t=1} to Theorem \ref{Thm:supersaturation-t=1 general}, which implies Theorem \ref{Thm:super-t=1-classify} clearly.

\begin{thm}\label{Thm:supersaturation-t=1 general}
Let $q$ be a large even integer and $G$ be a graph on $q^2+q+1$ vertices with $\frac{1}{2}q(q+1)^2+1$ edges.
Then either $G$ has at least $q^{9/8}/30 $ copies of $C_4$,
or there exists an orthogonal polarity graph $H$ of order $q$ such that $|E(G)\backslash E(H)|=s$ and $|E(H)\backslash E(G)|=s-1$ for some $1\leq s\leq q^{1/8}/30$.
In the latter case, the number of copies of $C_4$ in $G$ is between $sq-s^2$ and $sq+s^2$.
\end{thm}

\noindent{\it Proof.}
We will prove this by following the proof of Theorem 1.3 in \cite{HMY20} closely.
It should be mentioned that most claims there can be generalized in this setting easily and thus, in these cases,
we often only mention the modified statements without providing many details.
For other cases where extra arguments are needed, we give self-contained proofs.

By $\# C_4$, we denote the number of copies of $C_4$ in $G$.
For $v\in V:=V(G)$, let $c(v)$ be the number of copies of $C_4$ containing $v$.
Throughout this proof, let $t=q^{1/8}/30$, we may assume that
\begin{equation}\label{equ:new C4}
\# C_4<tq \mbox{ ~and~ } \sum_{v\in V}c(v)=4\cdot \# C_4<4tq.
\end{equation}

We say a pair $\{u,v\}\subseteq V$ is {\it opposite} if $d(u,v)\geq 2$.
The deficiency $f(v)$ is defined by $f(v)=\max\{q+1-d(v),0\}.$
So $f(v)+d(v)\geq q+1$ holds for every $v\in V$.

Recall the definitions of sets $UP$ and $P_2$.
For $A\subseteq V$, let $UP\cap A$ be the set of uncovered pairs $\{u,v\}\subseteq A$ of $G$
and let $P_2\cap A$ be the set of $2$-paths of $G$ with both endpoints in $A$.

First, we can get the following claim about the lower bound of $\# C_4$ (see Proposition 2.5 in \cite{HMY20}).

\begin{claim}[\cite{HMY20}]\label{Claim:C4}
	Let $G$ be a graph with $A\subseteq V(G)$. Then $2\# C_4\geq |P_2\cap A|+|UP\cap A|-\binom{|A|}{2}$.
\end{claim}

Similarly as in earlier sections, we let $S_i$ be the set of all vertices of degree $i$ in $G$ and let $S=\cup_{i=0}^q S_i$.
For $v\in V$, let $d_0(v)$ be the number of vertices $u\in V$ with $d(u,v)=0$. We retain Claim 4.1 in \cite{HMY20}.

\begin{claim}[\cite{HMY20}]\label{Claim:c(v)}
	Any vertex $v$ satisfies $c(v)\geq (d(v)-q-1)q- f(N(v))+d_0(v)$, and if $d(v)\geq q+1$ and $N(v)\cap S= \emptyset$ then $c(v)\geq 1$.
\end{claim}

By adjusting the proof of Claim 4.2 in \cite{HMY20}, one can derive $\Delta(G)\le q+2+t$.
Indeed, suppose on the contrary that there is a vertex $v_1\in V$ with $d(v_1)=q+k$ for some $t+3\leq k\leq q^2$.
Following the proof therein, we can get $\sum_{i=2}^{n} a_i\leq  q^2+(t+1)q$ and
\begin{equation*}
2\cdot \# C_4\geq  q^3-2q^2+2q-(q-1)\sum_{i=2}^{ n} a_i-0.5k^2+k(q^2-q+1.5)-2.
\end{equation*}
So we see that when $k=t+3$ and $\sum_{i=2}^{n} a_i=q^2+(t+1)q$,
the above inequality achieves its minimum $  q^2-0.5t^2-1.5t-2$,
which is a contradiction to \eqref{equ:new C4} for large $q$.
This proves $\Delta(G)\le q+2+t$.
We then see that the deficiency of $V$ is
\begin{equation}\label{equ:new Super-f(V)}
f(V)=(q+1)n+\sum_{k=0}^{t}(k+1)|S_{q+2+k}|-2e(G)=q-1+\sum_{k=0}^{t}(k+1)|S_{q+2+k}|.
\end{equation}
Let $S'=S_{q+2}\cup\cdots\cup S_{q+2+t}$.
By a similar proof as Claim 4.3 in \cite{HMY20}, one can show $|S'|\le 3\sqrt{tq}$.
Following the deductions of Claim 4.3 in \cite{HMY20},
one can conclude $q-4t\leq |S|\leq q-1+3(t+1)\sqrt{tq}$ and
\begin{equation}\label{equ:new Super-f(T)}
f(T)\le|S\cap T|+3(t+1)\sqrt{tq}+4t-1 \mbox{~~ for any ~} T\subseteq V.
\end{equation}

We also can derive the following analogue of Claim 4.3 in \cite{HMY20}.

\begin{claim}\label{Clm:A}
	$\sum_{k=0}^{t}(k+1)|S_{q+2+k}|\le 2t+2$.
\end{claim}

\begin{proof}
For any $u,v\in V$, we have $d(u,v)\leq  2\sqrt{ tq}$;
as otherwise there are at least $2tq$ copies of $C_4$, a contradiction to \eqref{equ:new C4}.
Choose $v_1,...,v_s\in S'$ for $s=\min\{4t,|S'|\}$. By the inclusion-exclusion principle,
\begin{align*}
|S|\geq |\bigcup_{j=1}^s(N(v_j)\cap S)|\geq \sum_{1\leq j\leq s} |N(v_j)\cap S|-\sum_{1\le i< j \le s}|N(v_i)\cap N(v_j)|
\ge  \sum_{1\leq j\leq s} |N(v_j)\cap S|-2\binom{s}{2}\sqrt{ tq}
\end{align*}
which implies that $\sum_{1\leq j\leq s} |N(v_j)\cap S|\le q+(3t+s^2+3)\sqrt{ tq}-1$.
Since each $C_4$ has two opposite pairs,
we see there are at least $\frac12\sum_{j=1}^s\sum_{i=s+1}^n \binom{d(v_j,v_i)}{2}$ copies of $C_4$ containing opposite pairs $\{v_j,v_i\}$ for $1\leq j\leq s$ and $s+1\leq i\leq n$.
For any $j\in [s]$, by the inequality (9) of Claim 4.1 in \cite{HMY20} and \eqref{equ:new Super-f(T)} we derive
\begin{align*}
& \sum_{i=s+1}^n\binom{d(v_j,v_i)}{2}\geq \sum_{i=s+1}^n (d(v_j,v_i)-1)\geq \sum_{i\in [n]\backslash \{j\}} (d(v_j,v_i)-1)- 2(s-1)\sqrt{ tq}\\
\geq & (d(v_j)-q-1)q -f(N(v_j))- 2(s-1)\sqrt{tq}\geq (d(v_j)-q-1)q-|N(v_j)\cap S|-X',
\end{align*}
where $X'=(3t+2s+1)\sqrt{tq}+4t-1$.
Putting the above together, we have that
\begin{align*}
2t q > & 2\#C_4\geq \sum_{j=1}^s\sum_{i=s+1}^n \binom{d(v_j,v_i)}{2}
\ge\sum_{j=1}^s(d(v_j)-q-1)q-\sum_{j=1}^s|N(v_j)\cap S|-sX'\\
\ge& q\sum_{k=0}^t(k+1)|S_{q+2+k}\cap \{v_1,...,v_s\}|-\big(q+(3t+s^2+3)\sqrt{ tq}-1\big)-sX'.
\end{align*}
Since $s/4\leq t\le  q^{1/8}/30$ and $q$ is large, we can derive $\sum_{k=0}^t(k+1)|S_{q+2+k}\cap \{v_1,...,v_s\}|\leq 2t+2$.
This shows that $s\leq 2t+2$ and by the choice of $s$, we have $S'=\{v_1,...,v_s\}$.
Therefore $\sum_{k=0}^{t}(k+1)|S_{q+2+k}|\le 2t+2$, completing the proof of Claim \ref{Clm:A}.
\end{proof}

Using Claim \ref{Clm:A} and \eqref{equ:new Super-f(V)}, we can easily deduce that $q-4t\leq |S|\leq f(V)\leq q+2t+1$ and for any $T\subseteq V$, $f(T)\leq |S\cap T|+6t+1$.
In particular, for any $v\in V$, we have $f(v)\leq 6t+2$ and $d(v)\geq q+1-f(v)\geq q-6t-1$.
Following the arguments in the proof of Claim 4.4 in \cite{HMY20}, we can show that for any $v\in V$, either $|N(v)\cap S|\leq 7t+2$ or $|N(v)\cap S|\geq q-11t-2$.
Again we note that there is at most one vertex (denoted by $z$ if it exists) in $V$ with $|N(z)\cap S|\geq q-11t-2$.
Let $W=S'\cup \{z\}.$
Claim \ref{Clm:A} shows $|W|\leq 2t+3$.

Let $c'(v)$ denote the number of vertices $x\in N(v)$ with $c(x)\geq 1$.
Then one can generalize the proof of Claim 4.5 in \cite{HMY20} and get the following:
if $\{u,v\}$ is an opposite pair with $u\in V\backslash W$ and $v\in V\backslash \{z\}$, then $c(u)+(6t+3)\cdot c'(v)\ge q-(13t+3)(6t+4)> q-200t^2.$
Choose $\ell= q^{1/4}/2$ such that $\ell q/(100t)-\binom{\ell}{2}\cdot2\sqrt{tq}>4tq$ holds for large $q$.
Then Claim 4.5 in \cite{HMY20} can be extended as following.

\begin{claim}\label{Clm:B}
    If there are $\ell$ opposite pairs $\{u_i,v_i\}$ for $i\in [\ell]$ such that $u_i, v_i\in V\backslash W$ and all $v_i$ are distinct,
    then there is some $u_i$ with $c(u_i)> 0.8q$.
\end{claim}

Let $A=\{v\in V: c(v)>0.8q \}$ and $X=A\cup W$.
So $|X|< {4tq}/{(0.8q)}+2t+3=7t+3$.
According to the arguments after Claim 4.5 in \cite{HMY20},
one can obtain that there exists an edge set $E^*\subseteq E(G)$ with $|E^*|\leq 45t^2$ such that $G'=G-E^*$ has at most $0.1q$ copies of $C_4$.

Having the above, the arguments after Claim 4.5 in \cite{HMY20} can be directly converted to show that there exists an orthogonal polarity graph $H$ of order $q$ such that $G'\subseteq H.$
Let $|E(G)\backslash E(H)|=s$. Then $|E(H)\backslash E(G)|=s-1$ and $1\leq s\leq |E(G)\backslash E(G')|=|E^*|\leq 45t^2$.
By Lemma \ref{Lem:polarity+plus}, there are at least $s(q-1)$ copies of $C_4$ in $H\cup (E(G)\backslash E(H))$ such that each edge in $E(H)\backslash E(G)$ can appear in at most $s$ of them.
This shows that $G$ has at least $s(q-1)-(s-1)s=sq-s^2$ copies of $C_4$.
If $s\geq t+1$, then since $s\leq 45t^2$, $G$ has at least $tq$ copies of $C_4$, contradicting our assumption.
Hence we have $1\leq s\leq t$.

It remains to show that the number $L$ of copies of $C_4$ in $G$ is at most $sq+s^2$.
Let $\mathcal{C}_0$ be the collection of all $C_4$'s in $G$ using exactly one edge in $E(G)\backslash E(H)$
and $\mathcal{C}_1$ be the collection of the remaining $C_4$'s in $G$.
By Lemma \ref{Lem:polarity+plus}, we have $|\mathcal{C}_0|\leq s(q+1)$.
There are three types of $C_4$'s in $\mathcal{C}_1$, namely using two edges, three edges or four edges in $E(G)\backslash E(H)$.
For each 4-cycle $C$ in $\mathcal{C}_1$, we define one or two pairs of edges in $(E(G)\cap E(C))\backslash E(H)$ as following.
If $C$ has exactly two edges in $E(G)\backslash E(H)$, then we take these two edges to form a pair.
If $C$ has three edges in $E(G)\backslash E(H)$ which form a path of length three, then we take the two non-incident edges to form a pair.
Otherwise all four edges in $C$ are from $E(G)\backslash E(H)$, then we take two pairs, each of which is formed by two non-incident edges of $C$.
Let us call all such pairs {\it feasible}.
It is easy to verify that each feasible pair can be contained in at most two 4-cycles in $\mathcal{C}_1$.
Thus we have $|\mathcal{C}_1|\leq 2\binom{s}{2}$, implying that $L=|\mathcal{C}_0|+|\mathcal{C}_1|\leq s(q+1)+s(s-1)=sq+s^2$.
This completes the proof of Theorem \ref{Thm:supersaturation-t=1 general}.
\QED

\section{Supersaturation: the general case}\label{sec:Super-general}
In this section, we establish two supersaturation results - Theorems \ref{Thm:supersaturation-general-t} and \ref{Thm:supersaturation-general-general-t},
which together imply Theorem \ref{Thm:supersaturation}.
We also give a proof of Proposition \ref{prop:1.2625}.

\subsection{A generalization of Theorem \ref{Thm:supersaturation-t=1}}
We now deduce a supersaturation result for a wider range on the number of edges from Theorem \ref{Thm:supersaturation-t=1 general}.
This is also optimal for infinitely many values $q$ (as powers of two).

\begin{thm}\label{Thm:supersaturation-general-t}
Let $q$ be a large even integer and $t$ be any integer such that $1\le t< q^{1/8}/30$.
Let $G$ be a graph on $q^2+q+1$ vertices with $\frac{1}{2}q(q+1)^2+t$ edges.
Then either $G$ has at least $(t+1)q-(t+1)^2$ copies of $C_4$,
or $G$ is obtained from an orthogonal polarity graph of order $q$ by adding $t$ new edges.
In particular, $G$ has at least $t(q-1)$ copies of $C_4$.
\end{thm}
\begin{proof}
Suppose on the contrary that $G$ has less than $(t+1)q-(t+1)^2$ copies of $C_4$.
We denote $G'$ to be any spanning subgraph of $G$ with $\frac{1}{2}q(q+1)^2+1$ edges.
Thus $G'$ has less than $(t+1)q-(t+1)^2< q^{9/8}/30$ copies of $C_4$.
By Theorem \ref{Thm:supersaturation-t=1 general}, there exists an orthogonal polarity graph $H$ of order $q$ such that $|E(G')\backslash E(H)|=s$ and $|E(H)\backslash E(G')|=s-1$ for some $1\leq s\leq t$.

Let $|E(G)\backslash E(H)|=j$.
Then $j-t=|E(H)\backslash E(G)|\leq |E(H)\backslash E(G')|=s-1\leq t$, implying that $t\leq j\leq 2t$.
So $G$ has at least $j(q-1)-(j-t)j$ copies of $C_4$.
If $j\geq t+1$, then $G$ has at least $(t+1)(q-1)-(t+1)$ copies of $C_4$, a contradiction.
So $j=t$ and $G$ is obtained from $H$ by adding $t$ new edges.
By Lemma \ref{Lem:polarity+plus}, $G$ has at least $t(q-1)$ edges.
This finishes the proof of Theorem \ref{Thm:supersaturation-general-t}.
\end{proof}

\subsection{A half-way bound}

\begin{thm}\label{Thm:supersaturation-general-general-t}
Let $q$ be a positive even integer.
If $G$ is a graph on $q^2+q+1$ vertices with $\frac{1}{2}q(q+1)^2+t$ edges for $t\geq 1$,
then $G$ contains at least $\frac12(tq-2.5q-t)$ copies of $C_4$.
\end{thm}

\noindent{\it Proof.}
As earlier, we let $\#C_4$ be the number of copies of $C_4$ in $G$,
$c(v)$ be the number of $C_4$ containing the vertex $v$, and $S_i$ be the set of vertices of degree $i$ in $G$.
We may assume that $\#C_4\leq \frac{1}{2}tq-\frac{1}{2}q$ (as otherwise we have the desired number of $C_4$'s).
Let $V=V(G)=\{v_1,...,v_{n}\}$, where $n=q^2+q+1$.

\begin{claim}\label{ClaimA}
Let $v_i\in V$ be any vertex with $d(v_i)=q+2+k$ for some $k\geq 0$.
Then $k\leq \frac{q}2$. Moreover, if $k=0$ then $c(v_i)>t$, and otherwise $c(v_i)>kq$.
\end{claim}

\begin{proof}
Without loss of generality, we consider $v_1$ with $d(v_1)=q+2+k$ for $k\geq 0$.
First suppose that $t\ge q+1$.
We have $\sum_{v\in V}d(v)=2e(G)=(q^2+q+1)q+(q^2+2t)$, where $q^2+2t\ge-(q^2+q+1)$.
By Claim \ref{Claim:C4} and Lemma \ref{Lem:CS-inquality}, we see $2\cdot\#C_4$ is at least
\begin{equation*}\label{dat}
\sum_{v\in V}\binom{d(v)}{2}-\binom{n}{2}\ge (q^2+q+1)\binom{q}{2}+q(q^2+2t)-\binom{q^2+q+1}{2}=2tq-q^2-q\ge tq.
\end{equation*}
This implies $\#C_4\ge \frac{1}{2}tq$, contradicting our assumption.
Hence we may assume that $t\leq q$.

Let $P_2'$ be the set of 2-paths in $G$ with none of its endpoints in $N(v_1)$.
Since each $C_4$ contains two covered pairs, we see that $2\cdot \# C_4\geq \sum_{\{u,v\}\subseteq V\backslash N(v_1)}\binom{d(u,v)}{2}+c(v_1)$ is at least
\begin{align*}
\sum\limits_{\{u,v\}\subseteq V\backslash N(v_1)}\big(d(u,v)-1\big)+c(v_1)\geq |P_2'|-\binom{n-d(v_1)}{2}+c(v_1).
\end{align*}
Let $a_i=|N(v_i)\cap N(v_1)|-1$ for $2\le i\le n$.
Then we have
\begin{equation*}\label{equ 1}
|P'_2|=\sum_{i=2}^{n}\binom{d(v_i)-a_i-1}{2} \mbox{ ~and~ } \sum_{i=2}^{ n} a_i\leq \sum_{i=2}^{ n}\binom{a_i+1}{2}\leq c(v_1)\leq tq/2\leq q^2/2.
\end{equation*}
On the other hand, we can derive that
\begin{equation*}
	\begin{split}
\sum\limits_{i=2}^{n}(d(v_i)-a_i-1)=&\big(2e(G)-d(v_1)\big)-\sum_{i=2}^{n}(a_i+1)=(q^2+q)(q-1)+X.
 \end{split}
	\end{equation*}
Here $X=q^2+2t-2-k-\sum_{i=2}^n a_i\geq q^2+2t-2-k-c(v_1)\geq-(q^2+q)$, where the inequalities hold because $1\leq t\leq q$, $k\le q^2-2$ and $\sum_{i=2}^{ n} a_i\leq c(v_1)\leq q^2/2$.
Putting the above all together, by Lemma \ref{Lem:CS-inquality}, we infer that $tq-q\geq 2\cdot\#C_4$ is at least
\begin{equation*}
\begin{split}
|P'_2|-\binom{n-d(v_1)}{2}+c(v_1)\geq &(q^2+q)\binom{q-1}{2}+(q-1)X-\binom{q^2-k-1}{2}+c(v_1)\\
\geq & k (q^2-q)-\frac{1}{2}(k^2+k)-c(v_1)(q-2) +2tq-q-2t+1.
\end{split}
\end{equation*}
Simplifying the above one can derive $c(v_1)\geq\frac{1}{q-2}\big(k(q^2-q)-\frac{1}{2}(k^2+k)+t(q-2)+1\big)$ and
$$c(v_1)-\frac{1}{2}tq\ge\frac{1}{q-2}\Big(k(q^2-q)-\frac{1}{2}(k^2+k)-\frac{1}{2}tq^2+2t(q-1)+1\Big):=\frac{g(k,t)}{q-2}.$$
Suppose $\frac{q}{2}< k\leq q^2-2$. Then $\frac{\partial g}{\partial t}(k,t)=-\frac{1}{2}(q-2)^2<0$ and $\frac{\partial^2 g}{\partial k^2}(k,t)=-1<0$.
This implies $\#C_4-\frac{1}{2}tq\ge c(v_1)-\frac{1}{2}tq\ge\min\{g(\frac{q}{2},q)/(q-2),g(q^2-2,q)/(q-2)\}>0$, a contradiction.
Therefore we obtain that $k\le\frac{q}{2}$.

If $k=0$, then $c(v_1)\ge(t(q-2)+1)/(q-2)>t$, as desired.
We also have
$$c(v_1)-kq\ge\frac{1}{q-2}\Big(kq-\frac{1}{2}(k^2+k)+t(q-2)+1\Big):=\frac{h(k,t)}{q-2}.$$
Consider $1\leq k\leq \frac{q}{2}$.
Since $\frac{\partial h}{\partial t}(k,t)=q-2>0$ and $\frac{\partial h}{\partial k}(k,t)=q-k-\frac{1}{2}>0$,
we have $c(v_1)-kq\geq h(1,1)/(q-2)>0$. This proves the claim.
\end{proof}

Let $\Delta(G)=q+2+m$. By Claim \ref{ClaimA}, we see that $m\leq\frac{q}{2}$.
	
\begin{claim}\label{claimB}
$|S_{q+2}|<2q$ and $\sum_{k=1}^{m}k|S_{q+2+k}|<2t$.
\end{claim}

\begin{proof}
By Claim \ref{ClaimA}, we see that if $v\in S_{q+2}$, then $c(v)>t$.
So $2tq> 4\#C_4>t|S_{q+2}|$, implying that $|S_{q+2}|<2q$.
For $v\in S_{q+2+k}$ with $1\leq k\leq m$, by Claim \ref{ClaimA} we have $c(v)>kq$.
Then $2tq> 4\#C_4>\sum_{k=1}^{m}kq\cdot|S_{q+2+k}|$, which implies that $\sum\limits_{k=1}^{m}k|S_{q+2+k}|<2t$.
\end{proof}

Now we are ready to complete the proof of Theorem \ref{Thm:supersaturation-general-general-t}.
First let us estimate $|UP|$, i.e., the number of uncovered pairs in $G$.
Consider a vertex $v\in S_{q+1}$, where the maximum degree of the vertices in $N(v)$ is at most $q+1$.
We assert that $v$ is contained in at least one uncovered pair.
Otherwise, as $q$ is even, the number $e_v$ of edges in $G[N(v)]$ is at least $q/2+1$ and
thus there are at most $(q+1)q-2e_v\leq q^2-2$ vertices adjacent to $N[v]$, which again forces an uncovered pair containing $v$.
For $1\le k\le q+1$, if $v\in S_{q+1-k}$ and the maximum degree of the vertices in $N(v)$ is at most $q+1$,
we can get at least
$$(n-1)-\sum_{u\in N(v)}(d(u)-1)=(q^2+2q+1-k)-\sum_{u\in N(v)}d(u)\ge kq$$
uncovered pairs containing $v$.
If some vertex in $N(v)$ has degree $q+1+\ell$ for $\ell\ge 1$,
then the above number of uncovered pairs containing $v$ will decrease by $\ell$. Thus, by double-counting, we can get
\begin{equation}\label{5}
	\begin{split}
	|UP|&\ge\frac{1}{2}\Big(|S_{q+1}|+\sum\limits_{k=1}^{q+1}kq|S_{q+1-k}|-\sum\limits_{u\in V}\sum\limits_{k=1}^{m+1}k|N(u)\cap S_{q+1+k}|\Big)\\
        &=\frac{1}{2}\Big(|S_{q+1}|+\sum\limits_{k=1}^{q+1}kq|S_{q+1-k}|-\sum\limits_{k=1}^{m+1}k(q+1+k)|S_{q+1+k}|\Big).
	\end{split}
   \end{equation}

Next, we give a lower bound on $|P_2|-\binom{n}{2}$. Let $S'=S_{q+2}\cup\cdots\cup S_{q+2+m}$.
We have
\begin{equation*}
\begin{split}
\sum_{v\in V\backslash S'}d(v)&=2e(G)-\sum\limits_{k=1}^{m+1}(q+1+k)|S_{q+1+k}|=\Big(n-\sum\limits_{k=1}^{m+1}|S_{q+1+k}|\Big)q+Y.
\end{split}
\end{equation*}
Here $Y=2t+q^2-\sum\limits_{k=1}^{m+1}(1+k)|S_{q+1+k}|\ge 2t+q^2-2|S_{q+2}|-3\sum\limits_{k=1}^{m}k|S_{q+2+k}|\geq q^2-4q-6t>-(n-|S'|)$,
where the second inequality holds because of Claim \ref{claimB}.
By Lemma \ref{Lem:CS-inquality},
\begin{equation*}
	\begin{split}
\sum_{v\in V\backslash S'}\binom{d(v)}{2}-\binom{n}{2}\ge&\Big(q^2+q+1-\sum\limits_{k=1}^{m+1}|S_{q+1+k}|\Big)\binom{q}{2}+qY-\binom{q^2+q+1}{2}\\
=&2tq-q^2-q-\sum\limits_{k=1}^{m+1}\Big(\binom{q}{2}+qk+q\Big)|S_{q+1+k}|.
\end{split}
   \end{equation*}
Recall that $P_2$ is the set of all 2-paths in $G$. We have
\begin{equation}\label{Equ:P2-n2}
	\begin{split}
	|P_2|-\binom{n}{2}&=\sum\limits_{k=1}^{m+1}|S_{q+1+k}|\binom{q+1+k}{2}+\sum_{v\in V\backslash S'}\binom{d(v)}{2}-\binom{n}{2}\\
        &\ge2tq-q^2-q+\frac{1}{2}\sum_{k=1}^{m+1}k(k+1)|S_{q+1+k}|.
	\end{split}
   \end{equation}
Let $M=2tq-q^2-q+\frac{1}{2}|S_{q+1}|$.
Combining \eqref{5} and \eqref{Equ:P2-n2}, by Claim \ref{Claim:C4} we derive that
  \begin{equation}\label{7}
	\begin{split}
     2\#C_4\ge&|P_2|+|UP|-\binom{n}{2}\geq M-\frac{q}{2}\Big(\sum\limits_{k=1}^{m+1}k|S_{q+1+k}|-\sum\limits_{k=1}^{q+1}k|S_{q+1-k}|\Big)\\
             =&M-\frac{q}{2}\big(2e(G)-(q+1)n\big)=tq-\frac{1}{2}(q^2+q-|S_{q+1}|).
      \end{split}
   \end{equation}
Finally, by Claim \ref{claimB} we have
\begin{equation*}
\begin{split}
q^2+q-|S_{q+1}|&=\sum\limits_{k=1}^{q+1}|S_{q+1-k}|+|S_{q+2}|+\sum\limits_{k=1}^{m}|S_{q+2+k}|-1\\
&\leq\sum\limits_{k=-(q+1)}^{m+1}\Big((q+1)-(q+1+k)\Big)\cdot|S_{q+1+k}|+2|S_{q+2}|+2\sum\limits_{k=1}^{m}k|S_{q+2+k}|-1\\
&\leq (q+1)n-2e(G)+4q+4t-1=5q+2t.
\end{split}
\end{equation*}
This together with \eqref{7} show that $\#C_4\geq \frac12(tq-2.5q-t)$.
The proof of Theorem \ref{Thm:supersaturation-general-general-t} is finished.
\QED

\medskip

We point out that Theorem \ref{Thm:supersaturation-general-general-t} only works for $t\geq 3$ and becomes invalid when $t\in \{1,2\}$.

\subsection{Proofs of Theorem \ref{Thm:supersaturation} and Proposition \ref{prop:1.2625}}

Before presenting the proofs, we show a upper bound on $h(q^2+q+1,t)$ for any prime power $q$ and $t\geq 1$,
using a random construction based on polarity graphs.

\begin{lem}\label{Lem:random-polarity}
Let $q$ be a prime power and $t$ be an integer such that $4t\leq q^3(q+1)$.
Then there exists a graph on $q^2+q+1$ vertices,
which contains at least $\frac12 q(q+1)^2+t$ edges and at most $500(tq+t^4/q^8)$ copies of $C_4$.
\end{lem}

\begin{proof}
We may assume $q\geq 3$ and $t\geq 1$. Let $H$ be an orthogonal polarity graph on $n=q^2+q+1$ vertices.
Let $\alpha=\frac{4t}{q^3(q+1)}\in [0,1]$ and let $G$ be obtained from $H$
by adding an edge for each non-adjacent pair of vertices independently and randomly with probability $\alpha$.
Denote by $X$ the number of new edges added to $H$.
Since the number of non-adjacent pairs in $H$ is $N=\binom{n}{2}-e(H)=\frac{q^3(q+1)}{2}$,
we have $E[X]=N\alpha=2t.$
Here, $X$ is a binomial random variable $X\sim Bin(N,\alpha)$.
Then the Chernoff bound states that $P(X<(1-\epsilon)N\alpha)\leq e^{-\epsilon^2 N\alpha/2}$.
Choosing $\epsilon=1/2$, we can get that $P(X<t)\leq e^{-\frac{t}{4}}\leq 0.78.$

Let $Y$ be the number of copies of $C_4$'s in $G$.
For $1\leq i\leq 4$, let $Y_i$ be the number of copies of $C_4$'s in $G$ consisting of exactly $i$ new edges.
We estimate $E[Y]=\sum_{i=1}^4 E[Y_i]$ as follows.
Note that every vertex in $H$ has degree $q$ or $q+1$.
For $E[Y_1]$, each of these $C_4$'s corresponds to a unique path of length three in $H$.
Thus we have $E[Y_1]\leq \frac12n(q+1) q^2\cdot \alpha\leq 3tq$.
For $E[Y_2]$, each of these $C_4$'s contains two edges in $H$ which are incident or not.
Thus we have $E[Y_2]\leq \frac12n^2(q+1)q\cdot \alpha^2+\binom{e(H)}{2}\cdot 2\alpha^2\leq 20t^2/q^2.$
Similarly, we can get that $E[Y_3]\leq e(H)\binom{n}{2}\cdot 2\alpha^3\leq 60t^3/q^5$
and $E[Y_4]\leq n^4\alpha^4/8\leq 50t^4/q^8.$
Since $t^2/q^2+t^3/q^5\leq tq+t^4/q^8,$
we can get that $E[Y]=\sum_{i=1}^4 E[Y_i]\leq 110(tq+t^4/q^8).$

Since $P(X\geq t)\geq 0.22$, we have $E[Y]=P(X\geq t)\cdot E[Y|X\geq t]+P(X<t)\cdot E[Y|X<t]\geq 0.22\cdot E[Y|X\geq t].$
So $E[Y|X\geq t]\leq E[Y]/0.22\leq 500(tq+t^4/q^8).$
This shows that there exists an $n$-vertex graph with at least $\frac12 q(q+1)^2+t$ edges
and at most $500(t\sqrt{n}+t^4/n^{4})$ copies of $C_4$'s.
\end{proof}

In aid of Lemma \ref{Lem:random-polarity}, we are ready to derive Theorem \ref{Thm:supersaturation} from Theorems \ref{Thm:supersaturation-general-t} and \ref{Thm:supersaturation-general-general-t}.

\medskip

\noindent {\bf Proof of Theorem \ref{Thm:supersaturation}.}
Let $q=2^k$ be sufficiently large.
So $\ex(q^2+q+1,C_4)=\frac12 q(q+1)^2$.

First we consider (A).
Suppose that $1\leq t< q^{1/8}/30$.
Let $G$ be a $(q^2+q+1)$-vertex graph with $\ex(q^2+q+1,C_4)+t$ edges.
By Theorem \ref{Thm:supersaturation-general-t} and Lemma \ref{Lem:polarity+plus}, $G$ has at least $t(q-1)$ copies of $C_4$,
with equality only if $G$ is obtained from an orthogonal polarity graph of order $q$ by adding $t$ edges between vertices of degree $q$.
As any two vertices of degree $q$ in a polarity graph has a common neighbor,
it is straightforward to see that when $G$ has exactly $t(q-1)$ copies of $C_4$,
these aforementioned $t$ new edges must form a matching (in fact it is also an induced matching by Lemma \ref{lem:Baer}).
This proves (A).

For the first assertion of (B),
it suffices to consider when $t\geq q^{1/8}/30$ and this follows from Theorem \ref{Thm:supersaturation-general-general-t} that $h(q^2+q+1,t)\geq \frac12(tq-2.5q-t)=\big(\frac12+o(1)\big)tq$, where $o(1)\to 0$ as $q\to \infty$.

Finally we prove the second assertion of (B) that $h(q^2+q+1,t)=\Theta(tq+t^4/q^8)$.
It is well known (see \cite{ES84}) that for any $c>0$ there exists some $c'>0$ such that $h(q^2+q+1,t)\geq c'\cdot t^4/q^8$ for any $t\geq cq^3$.
Also as $q$ is large, from the above proof, we have $h(q^2+q+1,t)\geq \big(\frac12+o(1)\big)tq\geq tq/3$ for any $t\geq 1$.
Note that $tq\geq t^4/q^8$ if and only if $t\leq q^3$.
Combining the above all together, we see that there exists some absolute constant $d>0$ such that for any $t\geq 1$,
$h(q^2+q+1,t)\geq d\cdot(tq+t^4/q^8)$.
The upper bound easily follows from Lemma \ref{Lem:random-polarity}.
The proof of Theorem \ref{Thm:supersaturation} is completed.
\QED

\medskip

We now complete the proof of Proposition \ref{prop:1.2625}.

\medskip

\noindent {\bf Proof of Proposition \ref{prop:1.2625}.}
Let $n$ be sufficiently large and $t\geq 3n^{1.2625}$.
By Proposition \ref{prop:lower-Turan}, there exists some prime $p$ with $\sqrt{n}-n^{0.2625}-1\leq p\leq \frac12(-1+\sqrt{4n-3})$ such that
$\ex(n,C_4)\geq \ex(p^2+p+1,C_4)\geq (n^{1.5}-3n^{1.2625}+n)/2.$
We first consider the lower bound of $h(n,t)$.
Consider any $n$-vertex graph $G$ with $\ex(n,C_4)+t$ edges.
Let $s$ be such that $\ex(n,C_4)+t=\frac{1}{2}(n^{1.5}+n)+s.$
So $s\geq t-\frac32n^{1.2625}\geq 0.5t.$
By Jensen's inequality, we see that the number $M$ of 2-paths in $G$ is at least
\begin{align*}
\sum_{v\in V(G)}\binom{d(v)}{2}
&\ge n\binom{ \frac{\sum d(v)}{n}}{2}=n\binom{ \frac{n^{1.5}+n+2s}{n}}{2}\geq \binom{n}{2}+2s\sqrt{n}+\frac{2s^2}{n}.
\end{align*}
As $s\geq 0.5t$, this implies that
$h(n,t)\geq \#C_4(G)\geq \frac12\Big(M-\binom{n}{2}\Big)\geq \frac{1}{2}\Big(\sum_{v\in V(G)}\binom{d(v)}{2}-\binom{n}{2}\Big)\geq \frac{t\sqrt{n}}{2}.$
It is known that if $t\geq n^{3/2}$, then $h(n,t)\geq c\cdot t^4/n^4$ holds for some $c>0$.
Combining these facts, we infer $h(n,t)=\Omega(t\sqrt{n}+t^4/n^4)$ for $t\geq 3n^{1.2625}$.
For the upper bound of $h(n,t)$, let $r$ be the integer such that $\ex(n,C_4)+t=\frac{1}{2}p(p+1)^2+r$.
By our choice of $p$, we can derive that $r\leq t+\frac32 n^{1.2625}\leq 2t$.
Using Lemma \ref{Lem:random-polarity}, there exists a graph on $p^2+p+1$ vertices with $p(p+1)^2/2+r=\ex(n,C_4)+t$ edges and at most $O(rp+r^4/p^8)=O(t\sqrt{n}+t^4/n^4)$ copies of $C_4$'s.
Since $p^2+p+1\leq n$, this also shows that $h(n,t)=O(t\sqrt{n}+t^4/n^4)$,
finishing the proof of Proposition \ref{prop:1.2625}. \QED

\section{Concluding remarks}\label{sec:concluding}

In this paper, we focus on extremal problems of $4$-cycles and prove several stability and supersaturation theorems.
These imply some exact or near-optimal extremal results on $C_4$ for infinite instances.
In what follows we discuss related problems, some of which in fact partially motivate the results here.

Theorem \ref{Thm:stability} provides a stability type result for dense $C_4$-free graphs $G$ on $q^2+q+1$ vertices where $q$ is even.
It states that if $e(G)\geq \frac12 q(q+1)^2-\frac12 q+o(q)$, then $G$ is contained in some (unique) polarity graph of order $q$.
We wonder if some other form of stability (e.g., in the sense of ``edit distance", which counts edges adding and deleting between $G$ and the extremal configuration)
can hold for a much weaker condition on the number of edges.
This stability also indicates that there exists some hierarchy on the number of edges for {\it maximal} $C_4$-free graphs $G$, in the interval starting from $\frac12 q(q+1)^2-\frac12 q+o(q)$.
This is because any of these graphs $G$ must be some polarity graph and according to Proposition \ref{prop:G(pi)}, $e(G)=\frac12 q(q+1)^2-\frac12 m\sqrt{q}$ holds for some integer $m\geq 0$.
We remark that using the result of Metsch \cite{Met92} and similar arguments here,
one also can establish an analogous stability result for $C_4$-free balanced bipartite graphs
where the size of two parts equals $q^2+q+1$ for any (even or odd) large integer $q$.
(As a side note here, we would like to mention that recently Nagy \cite{N19} proved some supersaturation results on 4-cycles in the bipartite setting.)

Arguably, Theorem \ref{Thm:stability} provides some (very weak) evidence to the following conjecture of McCuaig.

\begin{conjecture}[McCuaig, 1985; see \cite{Fur88,FS}]\label{Conj:McC}
Each extremal graph which achieves the maximum number $\ex(n,C_4)$ is a subgraph of some polarity graph.
\end{conjecture}

Suggested by orthogonal polarity graphs, the following supersaturation problem seems plausible:
For large $q=2^k$, is it true that $h(q^2+q+1,t)=t(q-1)$ holds for every $1\leq t\leq \frac{q}2$?
Theorem \ref{Thm:supersaturation} confirms this for $1\leq t\leq O(q^{1/8})$, while its proof perhaps can be generalized further.
Another problem is to determine all integers $t\geq 1$ such that every graph achieving the maximum $h(q^2+q+1,t)$ contains an orthogonal polarity graph of order $q$.
Regarding to Proposition \ref{prop:1.2625}, a similar supersaturation result for $C_4$ perhaps can hold under a more general condition:
Is there a constant $t_0$ such that $h(n,t)=\Theta(t\sqrt{n}+t^4/n^4)$ holds whenever $t\geq t_0$?

\bigskip

\noindent {\bf Acknowledgements.}
The authors would like to thank Prof. Klaus Metsch for sending a copy of \cite{Met92} and Prof. Lilu Zhao for sharing the result of \cite{BHP}, which are helpful for the study here.
The authors also would like to thank Prof. Zoltan F\"uredi for instructive comments.

\bibliographystyle{unsrt}

\end{document}